\documentclass[12pt,english,reqno]{amsart}

\usepackage{geometry}
\usepackage{amssymb,amsthm,amsmath}
\usepackage[dvips]{graphicx}
\usepackage{color}
\usepackage{enumerate}
\usepackage{comment}
\usepackage{textcomp}
\usepackage{bbm}
\usepackage[svgnames]{xcolor}
\usepackage{tikz}
\usetikzlibrary{decorations.pathreplacing, positioning}
\usetikzlibrary{shapes.geometric, arrows}
\usepackage{caption}
\usepackage[bookmarksnumbered=true]{hyperref} 
\hypersetup{
	colorlinks = true,
	linkcolor = Blue,
	anchorcolor = blue,
	citecolor = Green,
	filecolor = blue,
	urlcolor = FireBrick
}

\newcommand{\supp}{\mathop{\mathrm{supp}}}
\newcommand{\C}{\mathbb C}

\newcommand{\eps}{\varepsilon}

\newcommand{\R}{{\mathbb R}}

\renewcommand{\H}{{\mathbb H}}
\newcommand{\la}{\langle}
\newcommand{\ra}{\rangle}

\newcommand{\N}{{\mathbb N}}
\newcommand{\G}{{\mathcal G}}
\newcommand{\Z}{{\mathbb Z}}

\renewcommand{\S}{{\mathbb S}}

\newcommand{\bfd}{\mathbf{d}}
\renewcommand{\d}{\partial}

\newtheorem{theorem}{Theorem}
\newtheorem{lemma}[theorem]{Lemma}
\newtheorem{proposition}[theorem]{Proposition}

\theoremstyle{definition}
\newtheorem{definition}[theorem]{Definition}
\newtheorem{remark}[theorem]{Remark}

\begin{document}

\title[DN Neural operators]{Quantitative approximation theorem of neural operators for the Dirichlet-to-Neumann map}
\date{}

\author[Kow]{Pu-Yun Kow}
\address{Department of Artificial Intelligence, Tamkang University,  New Taipei City 251301, Taiwan.}
\email{169016@o365.tku.edu.tw}
\author[Aspri]{Andrea Aspri}
\address{Mathematics Department “F. Enriques”, Universit\`a degli Studi di Milano, Milano, Italy}
\email{andrea.aspri@unimi.it}
\author[Francini]{Elisa Francini}
\address{Dipartimento di Matematica e Informatica ''Ulisse Dini'', Universit\`a degli Studi di Firenze, Firenze, Italy}
\email{elisa.francini@unifi.it}
\author[Wang]{Jenn-Nan Wang}
\address{Institute of Applied Mathematical Sciences, National Taiwan University, Taipei 106, Taiwan.}
\email{jnwang@math.ntu.edu.tw}

\begin{abstract} 
In this paper, we study neural operators for the Dirichet-to-Neumann map (DN map), which is encountered in the famous Calder\'on's problem. In many numerical methods used for reconstructing the unknown conductivity from the knowledge of the DN map need to evaluate the corresponding DN map at a given conductivity function, which involves solving the boundary value problem for the conductivity equation and taking the Neumann data. Such procedure is time consuming and not effective. Hence, it is a favorable strategy to construct a surrogate to replace the forward map mapping conductivity functions to corresponding DN maps. The main theme of this paper is to derive quantitative bounds on the parametric complexity of these neural operators. In particular, under suitable assumptions, their depth grows at most logarithmically with $\varepsilon^{-1}$, while the relevant spectral and projection ranks grow algebraically. It should be noted that general neural operators suffer from "the curse of parametric dimensionality". This work demonstrates the construction of neural operators for the DN map that does not suffer from such curse.   
\end{abstract}


\maketitle

\tableofcontents

\begin{sloppypar}

\section{Introduction}\label{Introduction}

Nonlinear operators mapping from a function space to another function space are ubiquitous in science. One huge class of such nonlinear operators arise from PDEs, either direct or inverse problems. Recently, neural network architectures aiming to learn nonlinear operators have received a lot of attentions. Such network architectures are termed "neural operators", which are powerful extensions of traditional neural networks from finite-dimensional spaces to infinite-dimensional ones, see, for example, \cite{anandkumar2020neural, bhattacharya2021model, kovachki2023neural, lu2019deeponet} to name a few. One useful domain of using neural operators is to solve problems arising from PDEs. While traditional PDEs solvers often require substantial computational resources and time due to high dimensionality or nonlinearity, once trained, neural operators serve as surrogate models, providing significantly faster inference compared to traditional numerical solvers. 

A fundamental question for any neural network architecture is the universal approximation property, which ensures the applicability of neural network architectures. Early important results could be found in \cite{cybenko1989approximation} (for functions) and  \cite{chen1995approximation} (for functionals and operators). For neural operators, universality has been established in a variety of architectures, e.g.,  for DeepONets \cite{lu2021learning, lanthaler2022error}, for FNOs \cite{kovachki2021universal}, and for ANOs \cite{lanthaler2023nonlocal}. Since neural operators map deal with infinite-dimensional spaces, achieving manageable amount of learnable parameters is crucial in practice. However, it is known that general neural operators, which are characterized only by their $C^r$- or Lipschitz regularity, suffer from the "curse of parametric complexity" \cite{lanthaler2023curse}. In this work, we show that the general curse can be avoided for the DN map under certain conditions. We believe that such result is useful in solving Calder\'on's problem in practice.

\subsection{Related works}\label{sub:1.1}

Our construction of neural operators for the DN map is similar to those introduced in \cite{kovachki2023neural}. A distinctive feature of neural operators introduced there is the discretization-invariant property. Some neural operator architectures related to \cite{kovachki2023neural} have been proposed, differing primarily in the choice of basis functions within non-local operators, for example, Graph Neural Operators \cite{li2020neural}, Fourier Neural Operators \cite{li2020fourier}, Wavelet Neural Operators \cite{tripura2022wavelet, gupta2021multiwavelet}, Spherical Fourier Neural Operators \cite{bonev2023spherical}, and Laplace Neural Operators \cite{chen2023learning}, etc. Other notable neural networks for learning operators were already introduced earlier such as DeepONet \cite{chen1995universal, lu2019deeponet} and PCA-Net \cite{bhattacharya2021model}, \cite{de2022cost}. These architectures have demonstrated empirical success as the surrogate models of simulators across a wide range of problems in PDEs, see \cite{takamoto2022pdebench}. Even in Calder\'on's problem, neural operators based DeepONets are rigorously shown to be good surrogators of the forward and inverse Calder\'on maps (DN map) \cite{castro2024calderon}. One of critical theoretical foundations of neural operators is the universal approximation theorems. This theorem establishes that, for any target operator and desired level of accuracy, there exists a neural operator with sufficiently large learnable parameters such that it approximates the target operators with desired accuracy.
This indicates the capacity of the class of neural operators to approximate a wide range of operators. 
Universal approximation theorems for operator learning were established for neural operators \cite{kovachki2021universal, kovachki2023neural, lanthaler2023nonlocal, kratsios2024mixture}, DeepONet \cite{lu2019deeponet, lanthaler2022error}, and PCA-Net \cite{lanthaler2023operator}. A nice summary about recent progress on neural operators is given \cite{kovachki2024operator}. 

Our objective here is to establish the quantitative universal approximation theorem, which provides the upper bounds for the numbers of learnable parameters required to achieve a given level of approximation accuracy. To put our study in perspective, as already mentioned above, \cite{lanthaler2023curse} discovered that the operator learning for general operators suffers from “the curse of parametric complexity” in which the size of learnable parameters grows exponentially as the desired approximation accuracy increases. This unfavorable complexity may make  neural operators unpractical in solving real problems. However, when we restrict general operators to some operators in some specific PDEs, 
quantitative approximation theorems having manageable number of learnable parameters without experiencing exponential growth in model complexity have been established recently. For instance, \cite{kovachki2021universal, lanthaler2023operator} developed quantitative approximation theorems for the Darcy and Navier-Stokes equations using Fourier neural operators and PCA-Net, respectively. Additionally, \cite{lanthaler2023curse} analyzed the Hamilton-Jacobi equations with Hamilton-Jacobi neural operators, while \cite{furuya2024quantitative} addressed nonlinear parabolic equations. In \cite{scaop2025}, neural operators for the scattering problem (mapping from the refractive index to the far-field pattern) was constructed and a quantitative approximation theorem was proved as well. On the other hand, results in \cite{chen2023deep, lanthaler2022error, marcati2023exponential}, considered quantitative approximation theorems using DeepONet for a wide range of PDEs, including elliptic, parabolic, and hyperbolic equations, while \cite{deng2021convergence} focused on advection-diffusion equations.

\subsection{Our results and contributions}\label{sub:1.2}

Here, inspired by the study of Calder\'on's problem, we investigate neural operators for the DN map. Precisely, we aim to construct surrogates for the map sending the conductivity function to the corresponding DN map, that is, the map $\Lambda:\gamma\to\Lambda_\gamma$. In the general setting, for a bounded smooth domain $\Omega\subset\R^d$ with $d\ge 2$, if 
\begin{equation}\label{condr}
\gamma\in L^\infty(\Omega)\quad\mbox{and}\quad\gamma(x)\ge c>0\quad \mbox{for a.e.}\;\; x\in\Omega,
\end{equation}
then $\Lambda_\gamma:H^{1/2}(\partial\Omega)\to H^{-1/2}(\partial\Omega)$. In other words, we would like to learn the map $\Lambda$. Unlike neural operators previously studied, the range of $\Lambda$ is a space of \emph{operators} acting on Sobolev spaces. It is important to keep in mind that neural operators constructed here are independent of the discretization or mesh used to represent $\gamma$ and any Dirichlet data feeding into the DN map $\Lambda_\gamma$, i.e., $\Lambda_\gamma(f)\in H^{-1/2}(\partial\Omega)$ for $f\in H^{1/2}(\partial\Omega)$. Since the Sobolev space $H^{-1/2}(\partial\Omega)$ consists of distributions, taking its pointwise values on $\partial\Omega$, that is, $\Lambda_\gamma(f)(y)$ for $y\in\partial\Omega$, does not make sense. Therefore, the training of our neural operators needs special care. 

The focus of this work is the quantitative approximation theorem in which we are interested in deriving upper bounds of the parametric complexity of our neural operators with respect to any given approximation error. The parametric complexity quantities are the depth of the layer, the width of each layer, and the rank of the low-rank approximation, etc. 
In addition to the investigation of the approximation error, we will also study the generalization error. Our neural operators demonstrate a small generalization error so that they can be generalized well to unseen data. The first idea of the proof is to project the DN map onto a finite-dimensional subspace \cite[Lemma~4]{abraham2020statistical}. It is important to note that such projection incurs errors that can be estimated in terms of the rank of the projection \cite[Lemma~4]{abraham2020statistical}. This fact is useful in deriving quantitative approximation theorems. So to achieve our goal, it suffices to construct neural approximators to low-rank approximations of the associated DN map. Another idea, which is borrowed from \cite{furuya2024quantitative, scaop2025}, is that we can construct solutions to the conductivity equation iteratively and have a good control of construction error rigorously under certain conditions. From here, we can estimate upper bounds of neural operators' depth. It turns out that our neural operators do not suffer from the \emph{curse of dimensionality}.  

We want to remark that neural operators for the direct and inverse Calder\'on’s maps have been discussed in \cite{castro2024calderon}. Different from the neural operators studied in this paper, \cite{castro2024calderon} is based on DeepONets. In \cite{castro2024calderon}, "qualitative" approximation theorems for neural operators were rigorously proved. However, upper bounds of parametric complexity of neural operators were not investigated there. We also refer the reader to \cite{abhishek2024deeponet} for another interesting study on the application of DeepONets to Calderón’s problem. A complementary approach to operator learning for Calder\'on's problem is developed by \cite{dHKLN25}. They extend the inverse map, expressed in terms of integral kernels of Neumann-to-Dirichlet operators, to a Hilbert space of boundary kernels while preserving its modulus of continuity.
This extension provides a framework for Fourier neural operator
approximation with error bounds accounting for perturbations of the
boundary data. Their analysis concerns the inverse map, whereas
the present work focuses on quantitative bounds on the parametric
complexity of neural operator approximations of the forward map
$\gamma \mapsto \Lambda_\gamma$.

On the other hand, like the discretization-invariance property of neural operators given in \cite{kovachki2023neural}, neural operators we constructed here also satisfy this property. Therefore, once our neural operators are trained, normally, at a coarse grid, they can be used for all grids. In other words, our neural operators have the "zero-shot super-resolution" phenomenon, see \cite[Page~42]{kovachki2023neural}. To demonstrate the efficiency of the DN neural operators constructed herein, we apply the Ensemble Kalman Inversion method to the inverse problem of reconstructing the conductivity from the DN map. 

A closely related line of research is the work of Alberti and Santacesaria on Calder\'on's problem with finitely many measurements 
\cite{AlbSant1, AlbSant2, AlbSant3, AlbSant4}. Their results establish uniqueness, Lipschitz stability, and 
reconstruction for the conductivity when the unknown is assumed to belong to a finite-dimensional manifold or subspace, and when only a 
finite number of boundary measurements of the DN map is available. While their approach is fully analytical, our work focuses on the 
construction and analysis of neural operators as surrogate models for the DN map.

\section{The Dirichlet-to-Neumann map}

In this section, we review some theoretical foundations of the DN map. Without loss of generality, we assume that $\Omega=B$, a ball centered at the origin. Consider the boundary value problem for the conductivity equation 
\begin{equation}\label{cond}
\left\{
\begin{aligned}
&\nabla\cdot(\gamma\nabla u)=0\quad\mbox{in}\quad B,\\
&u=f\quad\mbox{on}\quad\partial B,
\end{aligned}\right.
\end{equation}
where the conductivity function $\gamma$ satisfies \eqref{condr}. We define the Dirichet-to-Neumann map $\Lambda_\gamma$ associated with \eqref{cond}, i.e., $\Lambda_\gamma:H^{1/2}(\partial B)\mapsto H^{-1/2}(\partial B)$ with
\[
\Lambda_\gamma(f)=\frac{\partial u_f}{\partial\nu}\Big{|}_{\partial B}
\]
where $u_f$ the solution of \eqref{cond} with boundary value $f$. It was known that the boundary value of $\gamma$ (including derivatives on the boundary whenever they exist) can be determined from $\Lambda_\gamma$, see for example, \cite[Chapter~3]{feldman2019calderon}. Therefore, for simplicity, we assume that $\supp(\gamma-1)\subset B$. It is convenient to consider 
\[
\tilde\Lambda_\gamma:=\Lambda_\gamma-\Lambda_1. 
\]
The famous Calder\'on's problem is to retrieve the conductivity $\gamma$ from $\Lambda_\gamma$ or $\tilde\Lambda_\gamma$. Calder\'on's problem is also known as the electric impedance tomography problem (EIT).  Here we are interested in constructing neural operators for $\Lambda:\gamma\mapsto\tilde\Lambda_\gamma$, which is an operator mapping from a function space to a space of $H^{1/2}(\partial B)-H^{-1/2}(\partial B)$ linear operators. 

\subsection{Low rank approximation}\label{sec2.1}

Starting from the results and ideas in \cite{abraham2020statistical}, which show that the operator $\tilde{\Lambda}_\gamma$ can be projected onto a finite-dimensional subspace while incurring only a small error, the goal of this section is to approximate this reduced model by means of neural operators. To begin,
%We would like to discuss a low rank approximation of $\tilde\Lambda_\gamma$ whose the aim of approximating such low rank approximation by neural operators. We will borrow ideas from \cite{abraham2020statistical}. 
denote $\{\phi_k=\phi_k^{(0)}: k\in \N\cup\{0\}\}$ be an orthonormal basis of $L^2(\partial B)$ consisting of real-valued eigenfucntions of the Laplace-Beltrami operator $\Delta_{\partial B}$. Let $\lambda_k$ be the corresponding eigenvalues:
\[
-\Delta_{\partial B}\phi_k=\lambda_k\phi_k.
\]
We now define $H^r(\partial B)$ for $r\ge 0$ as
\[
H^r(\partial B)=\{f\in L^2(\partial B)\;\mbox{such that}\; \sum_{k=0}^\infty(1+\lambda_k)^r|\la f,\phi_k\ra_{L^2(\partial B)}|^2:=\|f\|_{H^r(\partial B)}^2<\infty\},
\]
If $f, g\in H^r(\partial B)$, the inner product of $\la f,g\ra_{H^r(\partial B)}$ is defined as 
\[
\la f,g\ra_{H^r(\partial B)}=\sum_{k=0}^\infty(1+\lambda_k)^r\la f,\phi_k\ra_{L^2(\partial B)}\la g,\phi_k\ra_{L^2(\partial B)}.
\]
Therefore, $\{\phi_k^{(r)}:=(1+\lambda_k)^{-r/2}\phi_k: k\in\N\cup\{0\}\}$ forms an orthonormal basis of $H^r(\partial B)$. For $r<0$, $H^r(\partial B)$ is defined as the completion of $L^2(\partial B)$ with the norm $\|\cdot\|_{H^r(\partial B)}$ (now using $r<0$). The parameter $r$, which is at our disposal, reflects how the signal-to-noise ratio varies with frequency, see comments on \cite[Page 170]{abraham2020statistical}. 

Note that $\phi_0$ is a constant function. We therefore define the $H^r(\partial B)/\C$ norm by $\|[f]\|_{H^r(\partial B)/\C}=\inf_{z\in\C}\|f-z\|$. For $[f]$, the equivalence class over $\C$ of a function $f\in H^r(\partial B)$, can be characterized as 
\[
\|[f]\|_{H^r(\partial B)/\C}=\sum_{k=1}^\infty(1+\lambda_k)^r|\la f,\phi_k\ra_{L^2(\partial B)}|^2.
\]
On the other hand, let $H^s_\diamond(\partial B):=\{g\in H^s(\partial B): \la g,1\ra_{L^2(\partial B)}=0\}$, i.e., $H^s_\diamond(\partial B)$ is the orthogonal complement of the constant, then each $[f]\in H^s(\partial B)/\C$ has a representative $g\in H^s_\diamond(\partial B)$ with $\|f\|_{H^s(\partial B)/\C}=\|g\|_{H^s(\partial B)}$. Using these function spaces, the DN map can be defined in a more general setting
\[
\Lambda_\gamma: H^{1+s}(\partial B)/\C\mapsto H^s_\diamond(\partial B).
\]
Let $A$ and $B$ be separable Hilbert spaces. Denote ${\mathcal L}(A,B)$ the space of bounded operators $A\to B$ and ${\mathcal L}_2(A,B)$ the space of Hilbert-Schmidt operators $A\to B$. Define the orthonormal basis $(b_{jk}^{(p,q)})$ of ${\mathcal L}_2(H^p, H^q)$ by
\[
b_{jk}^{(p,q)}(f)=\phi_j^{(p)}\otimes\phi_k^{(q)}(f)=\la f,\phi_j^{(p)}\ra_{H^p}\phi_k^{(q)},\quad j,k\in\N.
\]
From \cite[Lemma 18]{abraham2020statistical}, for $p, q, r, s\in\R$ satisfying $p\le r$ and $q\ge s$, let $T\in {\mathcal L}(H^{p-(d-1)}, H^q)$, then one can show that $T\in{\mathcal L}_2(H^r,H^s)$ and for some constant $C$, depending only on $B$ and the difference $r-p, q-s$, such that
\begin{equation*}%\label{est0}
\|T\|_{{\mathcal L}_2(H^r,H^s)}\le C\|T\|_{{\mathcal L}(H^{p-(d-1)}, H^q)},
\end{equation*}
and moreover
\begin{equation}\label{est1}
\|T-\pi_{J,K}T\|_{{\mathcal L}_2(H^r,H^s)}\le C\|T\|_{{\mathcal L}(H^{p-(d-1)}, H^q)}\max\{(1+J^{1/(d-1)})^{p-r}, (1+K^{1/(d-1)})^{s-q}\},
\end{equation}
where
\begin{equation}\label{pjkt}
\pi_{J,K}T=\sum_{1\le j\le J}\sum_{1\le k\le K}\la T\phi_j^{(r)},\phi_k^{(s)}\ra_{H^s(\partial B)}b_{jk}^{(r,s)}.
\end{equation}

Here we will consider a more specific situation where ${p\le r, s=0}$ and denote $b_{jk}^{(r)}:=b_{jk}^{(r,0)}:H^r(\partial B)\to L^2(\partial B)$ the tensor product operator 
\[
b_{jk}^{(r)}(f)=\phi_j^{(r)}\otimes\phi_k^{(0)}(f)=\la \phi_j^{(r)},f\ra_{H^r(\partial B)}\phi_k^{(0)},\;\; f\in H^r(\partial B)/\C.
\]
We then define the space of Hilbert-Schmidt operators
\[
\H_r=\{T:H^r(\partial B)\to L^2(\partial B): T=\sum_{j,k=1}^\infty t_{jk}b_{jk}^{(r)}, t_{jk}\in\R, \sum_{j,k=1}^\infty t_{jk}^2<\infty\}.
\]
For $T\in \H_r$, we define a low rank approximation of $T$ as in \eqref{pjkt}, i.e.,
\begin{equation}\label{pjkr0}
\pi_{JK}T=\sum_{1\le j\le J}\sum_{1\le k\le K}\la T\phi_j^{(r)},\phi_k^{(0)}\ra_{L^2(\partial B)}b_{jk}^{(r)}.
\end{equation}
Let $m\in(0,1)$ and $B'$ be compactly contained in $B$. We introduce the parameter space
\[
S_{m,B'}=\{\gamma\in C(B): \inf_{x\in B}\gamma(x)\ge m,\;\gamma=1\;\mbox{on}\; B\setminus B'\}.
\]
Recall from \cite[Lemma 4]{abraham2020statistical} that if $\gamma\in S_{m,B'}$ and it is bounded by $M$ in $B$, then for any $\nu>0$ there exists a constant $C=C(\nu,B,B',r)>0$ such that
\begin{equation}\label{est2}
\|\tilde\Lambda_\gamma-\pi_{JK}\tilde\Lambda_{\gamma}\|_{\H_r}\le C\frac{M}{m}\min\{J,K\}^{-\nu}. 
\end{equation}
To derive \eqref{est2}, we use \eqref{est1} with $s=0$, $p=r-\nu(d-1)$, $q=\nu(d-1)$ and note that $\|\tilde\Lambda_\gamma\|_{H^{p-(d-1)}\to H^q}\le C\frac{M}{m}$. Without loss of generality, we assume $J\le K$ and choose $\nu>1$. Then if we want 
\begin{equation}\label{error1}
\|\tilde\Lambda_\gamma-\pi_{JK}\tilde\Lambda_{\gamma}\|_{\H_r}\le\eps,
\end{equation}
then it suffices to take 
\begin{equation}\label{J}
J=\left\lceil\left(C\frac Mm\frac{1}{\eps}\right)^{{1}/{\nu}}\right\rceil.
\end{equation}
Roughly speaking, $J$ increases sub-polynomially with respect to $\eps^{-1}$. In what follows, we choose $\nu=2$ and 
\begin{equation}\label{JJ}
J\sim\eps^{-1/2}.
\end{equation}

\subsection{Coefficients of low rank approximation}\label{sec:coeff} In view of \eqref{error1}, it is enough to construct neural operators approximating the low rank approximation $\pi_{JK}\tilde\Lambda_\gamma$. As shown in Section~2.1, for any $J\in\mathbb{N}$ the normalized DN map
$\tilde\Lambda_\gamma : H^{r}(\partial B)\to L^{2}(\partial B)$ admits the 
low-rank approximation
\[
\pi_{JK}\tilde\Lambda_\gamma
  = \sum_{j=1}^{J}\sum_{k=1}^{K}
    \big\langle \tilde\Lambda_\gamma\phi^{(r)}_{j},
    \phi^{(0)}_{k} \big\rangle_{L^{2}(\partial B)}\, b^{(r)}_{jk},
\]
where $\{\phi^{(r)}_{j}\}_{j\ge 0}$, $\{\phi^{(0)}_{k}\}_{k\ge 0}$ and $b^{(r)}_{jk}$ are the orthonormal 
bases and the tensor product operator introduced in Section~2.1, respectively.
Hence, in order to approximate $\tilde\Lambda_\gamma$ it is enough to approximate 
its coefficients
\begin{equation}\label{eq:piLambda}
c^{(r)}_{jk}(\gamma)
  := \big\langle \tilde\Lambda_\gamma\phi^{(r)}_{j},
     \phi^{(0)}_{k}\big\rangle_{L^{2}(\partial B)}.
\end{equation}
The goal of this subsection is therefore to compute these coefficients in a form 
suitable for neural network approximation. By using the representation of the DN map 
and Green's identity, 
\[
\begin{aligned}
c_{jk}^{(r)}(\gamma)&=\la \tilde\Lambda_\gamma\phi_j^{(r)},\phi_k^{(0)}\ra_{L^2(\partial B)}\\
&=\int_{\partial B}\phi_k^{(0)}(\Lambda_\gamma-\Lambda_1)\phi_j^{(r)}d\sigma\\
&=\int_B(\gamma-1)\nabla u^{(r)}_j\cdot\nabla v^{(0)}_k dx
\end{aligned}
\]
for $1\le j\le J,\, 1\le k\le K$, where $u_j^{(r)}$ and $v_k^{(0)}$ are solutions to
\begin{equation}\label{ujr}
\nabla\cdot(\gamma\nabla u_j^{(r)})=0\;\;\mbox{in}\;\; B,\quad u_j^{(r)}=\phi_j^{(r)}\;\;\mbox{on}\;\; \partial B,
\end{equation}
and
\begin{equation}\label{uk0}
\Delta v_k^{(0)}=0\;\;\mbox{in}\;\; B,\quad v_k^{(0)}=\phi_k^{(0)}\;\;\mbox{on}\;\; \partial B,
\end{equation}
respectively. Observe that if $u_j^{(0)}$ satisfies \eqref{ujr} for $r=0$, then $u_j^{(r)}=(1+\lambda_j)^{-r/2}u_j^{(0)}$. Here our aim is to construct an ML model to learn the operator 
\begin{equation}\label{dncoeff}
\gamma\mapsto c_{jk}^{(r)}(\gamma)=\int_B(\gamma-1)\nabla u^{(r)}_j\cdot\nabla v^{(0)}_k dx.
\end{equation}

To begin, we assume that $\gamma$ has a better regularity, say $\gamma\in W^{2,\infty}(B)$ (thus $\Delta\gamma\in L^\infty(B)$), and positivity of $\gamma$ remains true (see \eqref{condr}). Let $q=\frac{\Delta\sqrt{\gamma}}{\sqrt{\gamma}}$ and $w$ be the solution of 
\begin{equation}\label{wq}
\Delta w-q w=0\;\;\mbox{in}\;\; B,\;\; w=\sqrt{\gamma}\phi_j^{(r)}=\phi_j^{(r)}\;\;\mbox{on}\;\;\partial B.
\end{equation}
Then, let $u=\gamma^{-1/2}w$, $u$ solves the conductivity equation with boundary value $\phi_j^{(r)}$ on $\partial B$. As above, it suffices to consider \eqref{wq} for $r=0$. Now we let $v^{(0)}_j$ be the harmonic function in $B$ having boundary value $\phi_j^{(0)}$ on $\partial B$ (see \eqref{uk0}), then writing $w=w_j^{(0)}+v^{(0)}_j$, we see that $w_j^{(0)}$ satisfies
\begin{equation}\label{wqj}
\Delta z-qz=qv^{(0)}_j\;\;\mbox{in}\;\; B,\;\; z=0\;\;\mbox{on}\;\;\partial B,
\end{equation} 
and set $w_j^{(r)}=(1+\lambda_j)^{-r/2}w_j^{(0)}$ and $v_j^{(r)}=(1+\lambda_j)^{-r/2}v_j^{(0)}$. Thus, $u_j^{(r)}=\gamma^{-1/2}(w_j^{(r)}+v_j^{(r)})$. It is more convenient to split the coefficient \eqref{dncoeff} into two terms, i.e.,
\begin{equation}\label{dnmc}
\begin{aligned}
&\int_B(\gamma-1)\nabla u^{(r)}_j\cdot\nabla v^{(0)}_k dx\\
=&\int_B(\gamma-1)\nabla(\gamma^{-1/2} w^{(r)}_j)\cdot\nabla v^{(0)}_k dx+\int_B(\gamma-1)\nabla(\gamma^{-1/2} v^{(r)}_j)\cdot\nabla v^{(0)}_k dx.
\end{aligned}
\end{equation}

We can denote the second term on the right hand side of \eqref{dnmc} as a map $\gamma\mapsto a_{jk}^{(r)}(\gamma)$ defined by
\begin{equation}\label{bjk}
\begin{aligned}
a_{jk}^{(r)}(\gamma)&=\int_B(\gamma-1)\nabla(\gamma^{-1/2} v^{(r)}_j)\cdot\nabla v^{(0)}_k dx\\
&=-\frac 12\int_B(\gamma-1)\gamma^{-3/2}v_j^{(r)}\nabla\gamma\cdot\nabla v^{(0)}_k dx+\int_B(\gamma-1)\gamma^{-1/2}\nabla v_j^{(r)}\cdot\nabla v_k^{(0)}dx\\
&=-\int_B(\nabla\gamma^{1/2}+\nabla\gamma^{-1/2})v_j^{(r)}\cdot\nabla v^{(0)}_k dx+\int_B(\gamma^{1/2}-\gamma^{-1/2})\nabla v_j^{(r)}\cdot\nabla v_k^{(0)}dx\\
&=-(1+\lambda_j)^{-r/2}\int_B(\nabla\gamma^{1/2}+\nabla\gamma^{-1/2})v_j^{(0)}\cdot\nabla v^{(0)}_k dx\\
&\qquad+(1+\lambda_j)^{-r/2}\int_B(\gamma^{1/2}-\gamma^{-1/2})\nabla v_j^{(0)}\cdot\nabla v_k^{(0)}dx.
\end{aligned}
\end{equation}
Our task now is to find an approximation of $a_{jk}^{(r)}$ within  any desired accuracy. First of all, let us denote $\gamma_1=\gamma^{1/2}$ and $\gamma_2=\gamma^{-1/2}$, which are considered as input functions. On the other hand, in machine learning, derivatives of a function can be computed exactly using {\bf Automatic Differentiation (AutoDiff)} based on the chain rule. Therefore, the vector-valued functions $\nabla\gamma^{1/2}=\nabla\gamma_1$ and $\nabla\gamma^{-1/2}=\nabla\gamma_2$ can be computed in the \emph{input layer} using {\bf AutoDiff}. 

Let $\{\Phi_n\}_{n=1}^\infty$ be an orthonormal basis of $L^2(B)$, consisting of the Dirichlet eigenfunctions $\{\Phi_n\}_{n=1}^\infty$ of $-\Delta$ in $B$, see Appendix~\ref{A1} for an explicit construction of such a basis. Let $\{0<\mu_1<\mu_2\le\mu_3\le\cdots\}$ be the corresponding eigenvalue of $\Phi_n$, i.e., $-\Delta\Phi_n=\mu_n\Phi_n$ in $B$ and $\Phi_n|_{\partial B}=0$. By Weyl's asymptotic formula, $\mu_n\simeq n^{2/d}$. For any $f(x)\in L^2(B)$, we can write
\[
f=\sum_{n=1}^\infty\la f,\Phi_n\ra_{L^2(B)}\Phi_n.
\]
In terms of Weyl's asymptotic formula, we have that if $f\in H^s(B)$, then
\begin{equation}\label{hsnorm}
\sum_{n=1}^\infty n^{\frac{2s}{d}}|\la f,\Phi_n\ra_{L^2(B)}|^2<\infty.
\end{equation}
With the help of \eqref{hsnorm}, for any $\delta>0$, there exists $I=I(\delta)\in\mathbb{N}$ satisfying
\begin{equation}\label{delta1}
I(\delta)=O(\delta^{-{d}/{s}})
\end{equation}
such that
\begin{equation}\label{in}
\Big{\|}f-\sum_{n=1}^I\la f,\Phi_n\ra_{L^2(B)}\Phi_n\Big{\|}_{L^2(B)}<\delta,
\end{equation}
where the constant in $I(\delta)$ of \eqref{delta1} depends on the size of $\|f\|_{H^s(B)}$, which is uniform in any fixed ball in $H^s(B)$. Indeed, we can write
\[
\begin{aligned}
\Big{\|}f-\sum_{n=1}^I\la f,\Phi_n\ra_{L^2(B)}\Phi_n\Big{\|}^2_{L^2(B)}&=\sum_{n=I+1}^\infty|\la f,\Phi_n\ra_{L^2(B)}|^2\\
&\le (I+1)^{-\frac{2s}{d}}\sum_{n=I+1}^\infty n^{\frac{2s}{d}}|\la f,\Phi_n\ra_{L^2(B)}|^2\\
&\le (I+1)^{-\frac{2s}{d}}\|f\|^2_{H^s(B)}
\end{aligned}
\]  
which gives \eqref{in} by choosing $I$ satisfying \eqref{delta1}.

We now prove the following error estimate.
\begin{proposition}\label{prop0}
Assume $r\ge0$ and $d=2,3$. Let $\gamma\in H^{s}(B)\cap S_{m,B'}$ with $s>1+\frac d2$. For any $\eps>0$, there exist 
\begin{equation}\label{IJ}
I_1=O(\eps^{-\kappa_1}),\;\;\text{and}\;\; I_2=O(\eps^{-\kappa_2}),
\end{equation}
where
\[
\kappa_1=\frac{d}{s-1}\left(\frac 32+\frac{\alpha_d+1/4}{d-1}\right),\quad\kappa_2=\frac{d}{s}\left(\frac 32+\frac{\alpha_d+3/4}{d-1}\right),\quad\alpha_2=0,\quad\alpha_3=\frac 18,
\]
such that
\begin{equation}\label{err0}
\begin{aligned}
\sum_{1\le j,k\le J}\Bigg{|}a_{jk}^{(r)}(\gamma)-&(1+\lambda_j)^{-r/2}\left(-\sum_{n=1}^{I_1}\la \nabla\gamma^{1/2}+\nabla\gamma^{-1/2},\Phi_n\ra_{L^2(B)}\cdot\int_B\Phi_nv_j^{(0)}\nabla v^{(0)}_k dx\right.\\
&\left.\hspace{10mm}+\sum_{n=1}^{I_2}\la \gamma^{1/2}-\gamma^{-1/2},\Phi_n\ra_{L^2(B)}\int_B\Phi_n\nabla v_j^{(0)}\cdot\nabla v^{(0)}_k dx\right)\Bigg{|}^2<C\eps^2
\end{aligned}
\end{equation}
where $C=C(B, m, \|\gamma\|_{H^s(B)})$ and $J=O(\eps^{-1/2})$. 
\end{proposition}
\begin{proof}
It suffices to consider $r=0$. Let $f=\gamma^{1/2}-\gamma^{-1/2}$ and consider
\begin{equation}\label{eq00}
\begin{aligned}
&\Bigg{|}\int_B f\nabla v_j^{(0)}\cdot\nabla v^{(0)}_k dx-\int_B\sum_{n=1}^{I_2}\la f,\Phi_n\ra_{L^2(B)}\Phi_n(x)\nabla v_j^{(0)}\cdot\nabla v^{(0)}_kdx\Bigg{|}^2\\
&\le\|f-\sum_{n=1}^{I_2}\la f,\Phi_n\ra_{L^2(B)}\Phi_n\|^2_{L^2(B)}\|\nabla v_j^{(0)}\cdot\nabla v_k^{(0)}\|^2_{L^2(B)}\\
&\le C\delta^2\|\nabla v_j^{(0)}\|^2_{L^4(B)}\|\nabla v_k^{(0)}\|^2_{L^4(B)}
\end{aligned}
\end{equation}
for some $I_2=I_2(\delta)$ as in \eqref{delta1}, where we have used \eqref{in}.

We now use \eqref{eq00} and \eqref{estw2} to obtain
\begin{equation*}%\label{eq000}
\begin{aligned}
&\sum_{1\le j,k\le J}\Bigg{|}\int_B \left(f-\sum_{n=1}^{I_2}\la f,\Phi_n\ra_{L^2(B)}\Phi_n(x)\right)\nabla v_j^{(0)}\cdot\nabla v^{(0)}_k dx\Bigg{|}^2\\
&\le C\delta^2 \left(\sum_{1\le j\le J}j^{2(\alpha_d+3/4)/(d-1)}\right)^2\le C\delta^2 J^{(4\alpha_d+2d+1)/(d-1)}.
\end{aligned}
\end{equation*}
In view of \eqref{JJ}, to make 
\[
C\delta^2 J^{(4\alpha_d+2d+1)/(d-1)}\le\eps^2,
\]
it suffices to take $\delta=O(\eps^{3/2+[(\alpha_d+3/4)/(d-1)]})$ and thus $I_2=\eps^{-(d/s)\{3/2+[(\alpha_d+3/4)/(d-1)]\}}$.

Similarly, for $f=-(\nabla\gamma^{1/2}+\nabla\gamma^{-1/2})$, with the help of \eqref{estw1} and \eqref{estw2}, we estimate
\begin{equation*}
\begin{aligned}
&\sum_{1\le j,k\le J}\Bigg{|}\int_B \left(f-\sum_{n=1}^{I_1}\la f,\Phi_n\ra_{L^2(B)}\Phi_n(x)\right)v_j^{(0)}\cdot\nabla v^{(0)}_k dx\Bigg{|}^2\\
&\le C\delta^2 \left(\sum_{1\le j\le J}j^{2(\alpha_d-1/4)/(d-1)}\right)\left(\sum_{1\le k\le J}k^{2(\alpha_d+3/4)/(d-1)}\right)\le C\delta^2 J^{(4\alpha_d+2d-1)/(d-1)}.
\end{aligned}
\end{equation*}
To make 
\[
C\delta^2 J^{(4\alpha_d+2d-1)/(d-1)}\le\eps^2,
\]
it suffices to take $\delta=O(\eps^{3/2+[(\alpha_d+1/4)/(d-1)]})$ and thus 
$$I_1=\eps^{-(d/(s-1))\{3/2+[(\alpha_d+1/4)/(d-1)]\}}.$$
\end{proof}
\begin{remark}
In formula \eqref{err0} above, we do not compute the integrals
\[
\int_B\Phi_nv_j^{(0)}\nabla v^{(0)}_k dx\quad\text{or}\quad \int_B\Phi_n\nabla v_j^{(0)}\cdot\nabla v^{(0)}_k dx
\]
explicitly. They are treated as learnable parameters in the projection layer, see \eqref{G2} below. 
\end{remark}

%%%%%%%%%%%%%%%%%%%%%%%%

\begin{comment}
\begin{remark}\label{rem5}
For future reference, to guarantee both \eqref{error1} and \eqref{err0}, it suffices to take 
\begin{equation}\label{ieps}
I=I(\delta)=I(\eps^2)\;\;\mbox{in}\;\;\eqref{in},
\end{equation}
and choose
\begin{equation}\label{jeps}
J=\lceil C\eps^{-1/2}\rceil
\end{equation}
with $C=C(m,M)$ (taking $\nu=2$ in \eqref{J}). 
\end{remark}
\end{comment}

Similarly, for the first term on the right-hand side of \eqref{dnmc}, we can write
\begin{equation}\label{1term}
\begin{aligned}
&\int_B(\gamma-1)\nabla(\gamma^{-1/2} w^{(r)}_j)\cdot\nabla v^{(0)}_k dx\\
=&-\frac 12\int_B(\gamma-1)\gamma^{-3/2}w_j^{(r)}\nabla\gamma\cdot\nabla v^{(0)}_k dx+\int_B(\gamma-1)\gamma^{-1/2}\nabla w_j^{(r)}\cdot\nabla v_k^{(0)}dx\\
=&-\frac 12\int_B(\gamma^{-1/2}-\gamma^{-3/2})w_j^{(r)}\nabla\gamma\cdot\nabla v^{(0)}_k dx-\int_B w_j^{(r)}\nabla(\gamma^{1/2}-\gamma^{-1/2})\cdot\nabla v_k^{(0)}dx\\
=&-\frac 12\int_B(\gamma^{-1/2}-\gamma^{-3/2})w_j^{(r)}\nabla\gamma\cdot\nabla v^{(0)}_k dx-\frac 12\int_B w_j^{(r)}(\gamma^{-1/2}+\gamma^{-3/2})\nabla\gamma\cdot\nabla v_k^{(0)}dx\\
=&-\int_B\gamma^{-1/2}w_j^{(r)}\nabla\gamma\cdot\nabla v^{(0)}_k dx=-2(1+\lambda_j)^{-r/2}\int_B(\nabla\gamma^{1/2})w_j^{(0)}\cdot\nabla v^{(0)}_k dx
\end{aligned}
\end{equation}
since $v_k^{(0)}$ is harmonic in $B$ and $w_j^{(r)}|_{\partial B}=0$. 

\subsection{Iterative solutions} 
We next discuss the solvability of \eqref{wqj}. To construct approximate neural operators, we hope to solve \eqref{wqj} based on the Picard iteration. In other words, we look solution $z$ of \eqref{wqj} as an infinite series $z=\sum_{k=1}^\infty z^{(k)}$, where
\begin{equation*}
\Delta z^{(1)}=qv_j^{(0)}\;\;\mbox{in}\;\; B,\;\; z^{(1)}=0\;\;\mbox{on}\;\;\partial B,
\end{equation*}
and for $k\ge 2$
\begin{equation*}
\Delta z^{(k)}=qz^{(k-1)}\;\;\mbox{in}\;\; B,\;\; z^{(k)}=0\;\;\mbox{on}\;\;\partial B.
\end{equation*}
Let $G(x,y)$ be the Green's function of $\Delta$ in $B$, i.e., $\Delta_x G(x,y)=\delta(x-y)$ in $B$ and $G(\cdot,y)|_{\partial B}=0$. Then we can write
\begin{equation}\label{z0}
z^{(1)}(x)=\int_B G(x,y)q(y)v_j^{(0)}(y)dy
\end{equation}
and for $k\ge 1$
\begin{equation}\label{zk}
z^{(k+1)}(x)=\int_BG(x,y)q(y)z^{(k)}(y)dy.  
\end{equation}

For the Poisson equation $\Delta u=f$ in $B$ with $u|_{\partial B}=0$, standard elliptic regularity implies that, for every
$f\in L^2(B)$, the unique weak solution
$u\in H_0^1(B)$ actually belongs to
$H^2(B)\cap H_0^1(B)$ and satisfies
\[
\|u\|_{H^2(B)}
\le
C_{\rm reg}\|f\|_{L^2(B)},
\]
where $C_{\rm reg}=C_{\rm reg}(d,B)>0$. For convenience, we set
\[
C_\Delta:=\max\{1,C_{\rm reg}\}.
\]
Then $C_\Delta\ge1$ and, in particular,
\[
\|u\|_{H^2(B)}
\le C_\Delta\|f\|_{L^2(B)}.
\]
Define
\[
\Psi_q[u](x):=\int_B G(x,y)q(y)u(y)\,dy.
\]
Then
\begin{equation}\label{h2est}
\|\Psi_q[u]\|_{H^2(B)}
\le C_\Delta\|qu\|_{L^2(B)}
\le C_\Delta\|q\|_{L^\infty(B)}\|u\|_{L^2(B)}.
\end{equation}
Moreover,
\begin{equation}\label{iteration}
z^{(\ell)}=\Psi_q^{[\ell]}[v_j^{(0)}]
\qquad (\ell\ge1).
\end{equation}
\begin{lemma}\label{lem1}
Assume that
\begin{equation}\label{small}
\eta=\eta(q):=C_\Delta\|q\|_{L^\infty(B)}<1.
\end{equation}
Then $z=\sum_{\ell=1}^\infty\Psi_q^{[\ell]}[v_j^{(0)}]$
converges in $H^2(B)$ and is the unique solution of \eqref{wqj}.
For every integer $k\ge0$,
\begin{equation}\label{remainder}
\left\|z-\sum_{\ell=1}^k\Psi_q^{[\ell]}[v_j^{(0)}]\right\|_{H^2(B)}
\le \frac{\eta^{k+1}}{1-\eta}\,
\|v_j^{(0)}\|_{L^2(B)}.
\end{equation}
\end{lemma}

Lemma~\ref{lem1} will play an essential role in deriving an upper bound of the depth of neural operators. Putting \eqref{bjk}, \eqref{1term}, and \eqref{iteration} together gives
\begin{equation}\label{cjkgamma}
\begin{aligned}
c_{jk}^{(r)}(\gamma)=&-2(1+\lambda_j)^{-r/2}\int_B\sum_{\ell=1}^\infty\Psi_q^{[\ell]}[v_j^{(0)}]\nabla\gamma^{1/2}\cdot\nabla v^{(0)}_k dx+a_{jk}^{(r)}(\gamma).
%&-\frac 12(1+\lambda_j)^{-r/2}\int_B(\gamma^{-1/2}-\gamma^{-3/2})v_j^{(0)}\nabla\gamma\cdot \nabla v^{(0)}_k dx\\
%&+(1+\lambda_j)^{-r/2}\int_B(\gamma^{1/2}-\gamma^{-1/2})\nabla v_j^{(0)}\cdot\nabla v_k^{(0)}dx.
\end{aligned}
\end{equation}
Now we want to take a closer look at one iteration step in \eqref{cjkgamma}, i.e.,
\[
\begin{aligned}
\Psi_q^{[1]}[u](x)=\Psi_q[u](x)&=\int_BG(x,y)\gamma^{-1/2}(y)\Delta\gamma^{1/2}(y)u(y)dy\\
&=\int_BG(x,y)P(\gamma^{-1/2},\Delta\gamma^{1/2},u)(y)dy,
\end{aligned}
\]
where $P(\gamma^{-1/2},\Delta\gamma^{1/2},u)=\gamma^{-1/2}\Delta\gamma^{1/2}u$, a polynomial of $\gamma^{-1/2}, \Delta\gamma^{1/2}, u$. Note that $\Delta\gamma^{1/2}$ can be computed by {\bf AutoDiff}. Since $G\in L^2(B\times B)$, we may write 
\[
G(x,y)=\sum_{n,m=1}^\infty c_{nm}\Phi_n(x)\otimes\Phi_m(y)
\]
for some constants $c_{nm}\in\R$. 
In fact, using the same idea as in \cite{scaop2025}, we can show that $G\in H^t(B\times B)$ for $0<t<2-\frac d2$, which is valid for $d=2,3$. Consequently, similar to \eqref{hsnorm}, we have
\begin{equation}\label{cnm}
\sum_{n,m=1}^\infty(n^{2/d}+m^{2/d})^t|c_{nm}|^2<\infty.
\end{equation} 
As above, by \eqref{cnm}, we have
\[
\begin{aligned}
&\Big{\|}G-\sum_{n,m=1}^Nc_{nm}\Phi_n\otimes\Phi_m\Big{\|}^2_{L^2(B\times B)}=\sum_{\{1\le n,m:\, \max\{m,n\}>N\}}|c_{nm}|^2\\
&\le(1+(N+1)^{2/d})^{-t}\|G\|^2_{H^t(B\times B)},
\end{aligned}
\]
which implies that, for any $\eps>0$, there exists an 
\begin{equation}\label{ntd}
N=N(\eps)=O(\eps^{-2d/t}) \quad\text{for any}\;\; t\in(0,2-\frac{d}{2}),
\end{equation}
such that
\begin{equation}\label{gn}
\Big{\|}G-\sum_{n,m=1}^Nc_{nm}\Phi_n\otimes\Phi_m\Big{\|}_{L^2(B\times B)}<\eps^2.
\end{equation}
We now denote $G_N(x,y):=\sum_{n,m=1}^Nc_{nm}\Phi_n(x)\otimes\Phi_m(y)$, 
\begin{equation}\label{qN}
\begin{aligned}
\Psi_{q,N}[u](x)&:=\int_BG_N(x,y)P(\gamma^{-1/2},\Delta\gamma^{1/2},u)(y)dy\\
&=\sum_{n,m=1}^Nc_{nm}\la\Phi_m,P(\gamma^{-1/2},\Delta\gamma^{1/2},u)\ra_{L^2(B)}\Phi_n(x),
\end{aligned}
\end{equation}
and $\Psi_{q,N}^{[\ell]}$ is defined accordingly. By \eqref{gn} and the Cauchy-Schwarz inequality,
\begin{equation}\label{ggn}
\begin{aligned}
\|\Psi_q[u]-\Psi_{q,N}[u]\|_{L^2(B)}
&\le \|G-G_N\|_{L^2(B\times B)}
      \|q\|_{L^\infty(B)}\|u\|_{L^2(B)}\\
&\le \eps^2\eta\|u\|_{L^2(B)},
\end{aligned}
\end{equation}
where we used $C_\Delta\ge1$.

To control the truncated operator with the same contraction constant,
let $A=\Delta_D^{-1}$ and let $P_N$ be the orthogonal projection of
$L^2(B)$ onto $\operatorname{span}\{\Phi_1,\ldots,\Phi_N\}$.
The integral operator with kernel $G_N$ is
\[
A_N=P_NAP_N.
\]
Consequently,
\[
\|A_N\|_{L^2(B)\to L^2(B)}
\le\|A\|_{L^2(B)\to L^2(B)}\le C_\Delta.
\]
Since $\Psi_q=AM_q$ and $\Psi_{q,N}=A_NM_q$, where $M_qu=qu$,
we obtain, for every $N\ge1$,
\begin{equation}\label{psiqn}
\max\bigl\{\|\Psi_q[u]\|_{L^2(B)},
            \|\Psi_{q,N}[u]\|_{L^2(B)}\bigr\}
\le\eta\|u\|_{L^2(B)}.
\end{equation}
In particular,
\begin{equation}\label{psiql}
\|\Psi_{q,N}^{[\ell]}[u]\|_{L^2(B)}
\le\eta^\ell\|u\|_{L^2(B)}
\qquad (\ell\ge1),
\end{equation}
and the same estimate holds for $\Psi_q^{[\ell]}$.

For $T=\Psi_q$ and $S=\Psi_{q,N}$, the telescoping identity
\[
T^\ell-S^\ell
=\sum_{h=0}^{\ell-1}T^{\ell-1-h}(T-S)S^h
\]
together with \eqref{ggn} and \eqref{psiqn} gives
\begin{equation}\label{psinell}
\|\Psi_q^{[\ell]}[u]-\Psi_{q,N}^{[\ell]}[u]\|_{L^2(B)}
\le\ell\eps^2\eta^\ell\|u\|_{L^2(B)}
\qquad (\ell\ge1).
\end{equation}
In particular,
\[
\left\|\sum_{\ell=1}^L
\bigl(\Psi_q^{[\ell]}[u]-\Psi_{q,N}^{[\ell]}[u]\bigr)
\right\|_{L^2(B)}
\le\eps^2\frac{\eta}{(1-\eta)^2}\|u\|_{L^2(B)}.
\]

\begin{proposition}\label{prop1}
Fix $\eta_0\in(0,1)$.
Let $\gamma\in W^{2,\infty}(B)\cap S_{m,B'}$ satisfy
\[
\eta(q)=C_\Delta\|q\|_{L^\infty(B)}\le\eta_0.
\]
Given $\eps\in(0,1)$, let $J$ be chosen as in \eqref{JJ},
so that
\[
J\le C\eps^{-1/2},
\]
where $C$ depends only on the a priori parameters.
For a fixed $t\in(0,2-d/2)$, choose $N$ as in \eqref{gn},
with $N=O(\eps^{-2d/t})$.
Then there exists an integer $L=L(\eps,\eta_0)$ such that
\begin{equation}\label{psin}
\begin{aligned}
\sum_{1\le j,k\le J}\Bigg|
\int_B\Bigg(
\sum_{\ell=1}^\infty\Psi_q^{[\ell]}[v_j^{(0)}]
-\sum_{\ell=1}^L\Psi_{q,N}^{[\ell]}[v_j^{(0)}]
\Bigg)
\nabla\gamma^{1/2}\cdot\nabla v_k^{(0)}\,dx
\Bigg|^2
\le C\eps^2.
\end{aligned}
\end{equation}
Here $C$ may depend on $d,B,B',r,m,\eta_0$ and $\|\gamma\|_{W^{2,\infty}(B)}$, but not on $\eps,J,L,N$.
One may take
\begin{equation}\label{depth}
L=\left\lceil
\frac{2\log(\eps^{-1})}{\log(\eta_0^{-1})}
\right\rceil
=O\bigl(\log(\eps^{-1})\bigr)
\quad\text{as }\eps\to 0.
\end{equation}
\end{proposition}
\begin{proof}
Set
\[
e_{j,L,N}:=
\sum_{\ell=1}^\infty\Psi_q^{[\ell]}[v_j^{(0)}]
-\sum_{\ell=1}^L\Psi_{q,N}^{[\ell]}[v_j^{(0)}].
\]
By \eqref{remainder}, \eqref{psinell}, and $\eta\le\eta_0$,
\[
\|e_{j,L,N}\|_{L^2(B)}
\le\left(
\frac{\eta_0^{L+1}}{1-\eta_0}
+\eps^2\frac{\eta_0}{(1-\eta_0)^2}
\right)\|v_j^{(0)}\|_{L^2(B)}.
\]
Introduce
\[
F_d(J):=
\begin{cases}
J^2(1+\log J),&d=2,\\
J^2,&d=3.
\end{cases}
\]
Estimates \eqref{estw20} and \eqref{estw21} imply
\[
\left(\sum_{j=1}^J\|v_j^{(0)}\|_{L^2(B)}^2\right)
\left(\sum_{k=1}^J\|\nabla v_k^{(0)}\|_{L^2(B)}^2\right)
\le C_{d,B}F_d(J).
\]
Hence, by Cauchy--Schwarz and $(a+b)^2\le2a^2+2b^2$,
\begin{equation}\label{eq:001}
\begin{aligned}
&\sum_{1\le j,k\le J}
\left|\int_B e_{j,L,N}
\nabla\gamma^{1/2}\cdot\nabla v_k^{(0)}\,dx\right|^2\\
&\quad\le
2C_{d,B}\|\nabla\gamma^{1/2}\|_{L^\infty(B)}^2F_d(J)
\left[
\frac{\eta_0^{2(L+1)}}{(1-\eta_0)^2}
+\frac{\eps^4\eta_0^2}{(1-\eta_0)^4}
\right].
\end{aligned}
\end{equation}
For $L$ chosen in \eqref{depth}, we have
$\eta_0^{L+1}\le\eps^2$.
By \eqref{J} with $\nu=2$, there exists a constant
$C_J>0$, depending only on the a priori parameters,
such that
\[
J\le C_J\eps^{-1/2}.
\]
Consequently,
\[
F_d(J)\le C_{C_J}\eps^{-1}
\bigl(1+\log(\eps^{-1})\bigr)
\qquad(d=2,3).
\]
Since
\[
\|\nabla\gamma^{1/2}\|_{L^\infty(B)}
\le\frac{1}{2\sqrt m}\|\nabla\gamma\|_{L^\infty(B)}
\]
and $\eps(1+\log(\eps^{-1}))\le1$ for $0<\eps<1$,
\eqref{eq:001} is bounded by
\[
C\eps^3\bigl(1+\log(\eps^{-1})\bigr)\le C\eps^2.
\]
\end{proof}

\section{Neural operators and parametric complexity}\label{sec3}

Employing  the quantitative approximation results of Section \ref{sec:coeff}, we construct a 
neural operator architecture whose design reflects the underlying analytical structure. Here, for numerical purposes, we assume the conductivity $\gamma$ to be more regular. Given constants $0<m\le M$ and a ball $B'\Subset B$, define
\[
{\mathcal S}^{M}_{m,B'}=
\{\gamma\in H^s(B):\gamma\ge m,\;
\|\gamma\|_{H^s(B)}\le M,\;
\gamma=1\text{ in }B\setminus B'\},
\]
where $s>2+d/2$.
By Sobolev embedding, ${\mathcal S}^{M}_{m,B'}\subset C^2(\bar B)$,
and this class is uniformly bounded in $W^{2,\infty}(B)$.

For a fixed $\eta_0\in(0,1)$, motivated by the uniform
contraction assumption in Proposition~\ref{prop1}, we define
\[
{\mathcal S}^{M,\eta_0}_{m,B'}:=
\left\{
\gamma\in{\mathcal S}^{M}_{m,B'}:
\eta(q)=C_\Delta\|q\|_{L^\infty(B)}\le\eta_0
\right\}.
\]
Thus, the smallness condition required in
Proposition~\ref{prop1} holds uniformly for every
$\gamma\in{\mathcal S}^{M,\eta_0}_{m,B'}$.
The architecture below is defined on ${\mathcal S}^{M}_{m,B'}$,
whereas the uniform approximation theorem will be stated on
${\mathcal S}^{M,\eta_0}_{m,B'}$.
By construction, both $\Psi_{q}$ and $\Psi_{q,N}$
have $L^2(B)$ operator norm at most $\eta_0$, uniformly in
$\gamma\in{\mathcal S}^{M,\eta_0}_{m,B'}$ and $N\ge1$.
Consequently, the constants in Proposition~\ref{prop1} can be
chosen uniformly on this class; they may depend on $\eta_0$
but not on the individual conductivity.

From \eqref{eq:piLambda}, choosing $K=J$, i.e. 
\begin{equation}\label{pjj}
\pi_{JJ}\tilde\Lambda_\gamma=\sum_{1\le j\le J}\sum_{1\le k\le J}\la \tilde\Lambda_\gamma\phi_j^{(r)},\phi_k^{(0)}\ra_{L^2(\partial B)}b_{jk}^{(r)}=\sum_{1\le j\le J}\sum_{1\le k\le J}c_{jk}^{(r)}(\gamma)b_{jk}^{(r)}
\end{equation}
the construction of the neural-operator consists of three components:
\begin{enumerate}
    \item a \emph{lifting step}, encoding the functions 
          $v^{(0)}_j$, $\Delta(\gamma^{1/2})$, $\gamma^{-1/2}$ and their derivatives;
    \item a sequence of \emph{hidden layers} combining pointwise transformations 
          with nonlocal integral operators of the type developed in  
          \cite{furuya2024quantitative, kovachki2023neural, lanthaler2023nonlocal};
    \item a final \emph{projection step}, assembling the outputs into the truncated
          coefficient array $\{c^{(r)}_{jk}(\gamma)\}_{1\le j,k\le J}$.
\end{enumerate}
We now give the precise definition of the neural operator
$\mathcal{G}$ and subsequently estimate its
parametric complexity.

%With all the preparation obtained above, we are now ready to give a formal architecture of neural operators for the normalized DN map $\tilde\Lambda_\gamma:H^r(\partial B)\to L^2(\partial B)$. Let $\{\phi_k^{(r)}: k\in\N\cup\{0\}\}$ be an orthonormal basis of $H^r(\partial B)$, which is described precisely in Section~\ref{sec2.1}. In particular, $\{\phi_k^{(0)}: k\in\N\cup\{0\}\}$ is an orthonormal basis of $L^2(\partial B)$. Furthermore, let $\{\Phi_n: n\in\N\cup\{0\}\}$ be an orthonormal basis of $L^2(B)$ (see Appendix~\ref{A1} for a straightforward construction). 
% Define the low rank approximation 
% \begin{equation}\label{pjj}
% \pi_{JJ}\tilde\Lambda_\gamma=\sum_{1\le j\le J}\sum_{1\le k\le J}\la \tilde\Lambda_\gamma\phi_j^{(r)},\phi_k^{(0)}\ra_{L^2(\partial B)}b_{jk}^{(r)}=\sum_{1\le j\le J}\sum_{1\le k\le J}c_{jk}^{(r)}(\gamma)b_{jk}^{(r)}
% \end{equation}
% as given in \eqref{pjkr0}, where
% \[
% b_{jk}^{(r)}(f)=\phi_j^{(r)}\otimes\phi_k^{(0)}(f)=\la f,\phi_j^{(r)}\ra_{H^r(\partial B)}\phi_k^{(0)},\quad j,k\in\N,
% \]
% for all $f\in H^r(\partial B)$. 
\begin{definition}[Neural operators, \cite{furuya2024quantitative, kovachki2023neural, lanthaler2023nonlocal}]\label{def:neural-operator}

Let ${\mathcal G} : {\mathcal S}^{M}_{m,B'} \to H^{1/2}(\d B)\to H^{-1/2}(\partial\Omega)\,\text{\it linear maps}$ by
\[
{\mathcal G} : \gamma \mapsto {\Gamma},
\]
where the output function ${\Gamma}$ is determined inductively as follows. Given $\gamma\in{\mathcal S}^{M}_{m,B'}$.

\bigskip

\noindent
{\bf For each $j=1:J$}

\medskip\noindent
{\rm 1}. {\bf (Lifting)} Let $v_j^{(0)}$ be defined in \eqref{uk0}. Define  $\gamma_1=\gamma^{1/2}$, $\gamma_2=\gamma^{-1/2}$. We then evaluate $\Delta\gamma_1$ using {\bf AutoDiff}. Define vector-valued function
\[
U_{j}^{(0)}:=\begin{pmatrix}v_j^{(0)}\\\gamma_2\\\Delta\gamma_1\end{pmatrix}\in L^2(B;\R)^{3}.
\] 

\medskip
\noindent
{\rm 2}. {\bf (Hidden Layers)} 
For $0 \le \ell \le {\mathbf L}-1$, $U_{j}^{(\ell+1)} = (\hat{u}^{(\ell+1)}_1, \hat{u}^{(\ell+1)}_2, \ldots, \hat{u}^{(\ell+1)}_{d_{\ell+1}}) \in L^2(B; \R)^{d_{\ell+1}}$ are iteratively given by 
\[
U_{j}^{(\ell+1)}(x) = \sigma \left( W^{(\ell+1)} U_{j}^{(\ell)}(x) + (K^{(\ell+1)}_{N_1}U_{j}^{(\ell)})(x) 
+ b^{(\ell+1)}
\right), \quad 0 \le \ell \le {\mathbf L}-1,
\]
where $d_{0}=3$, and $\sigma : \mathbb{R} \to \mathbb{R}$ is a nonlinear activation operating element-wise, and $W^{(\ell+1)} \in \mathbb{R}^{d_{\ell+1}\times d_{\ell}}$ is a weight matrix of the $(\ell+1)$-th hidden layer, $b^{(\ell+1)} \in\R^{d_{\ell+1}}$ are biases, and the integral operator $K^{(\ell+1)}_{N_1} :L^{2}(B; \R)^{d_{\ell}} \to L^{2}(B; \R)^{d_{\ell + 1}}$ are defined by
\begin{align*}
& 
(K^{(\ell+1)}_{N_1}u)(x) :=
\sum_{1\le m,n\le N_1}
c_{nm}^{(\ell+1)} \langle \Phi_{m}, u \rangle
\Phi_{n}(x) \quad \text{with $C^{(\ell+1)}_{n,m} \in \mathbb{R}^{d_{\ell+1} \times d_{\ell}}$}.
\end{align*}
Here, we denote the inner product 
\[
\langle \Phi_{m}, U \rangle
:=
\left( \langle \Phi_{m}, \hat u_1 \rangle_{L^2(B;\R)},\ldots, \langle \Phi_{m}, \hat u_{d_\ell} \rangle_{L^2(B;\R)} \right)^T \in \mathbb{R}^{d_{\ell}},
\]
for $U = (\hat u_1,\dots,\hat u_{d_\ell}) \in L^2(B;\R)^{d_{\ell}}$. Next, we define 
\[
X^{(\ell+1)}_j:=D_j^{(\ell+1)}U_j^{(\ell+1)},
\]
where $D_j^{(\ell+1)}\in\R^{1\times d_{\ell+1}}$. Finally, we sum $X_j$'s up to $\ell={\mathbf L}-1$, i.e.  $\sum_{i=1}^{{\mathbf L}}X_j^{(i)}$. 

\medskip\noindent
{\bf For each $j, k=1:J$}

\medskip\noindent
{\rm 3}. {\bf (Projection)} The output $\Gamma$ is given by the sum of two parts as follows. Given $v_k^{(0)}$ and compute $\nabla v_k^{(0)}$. Define
\[
\Gamma^{(1)}_{jk}:=\la \nabla\gamma_1\cdot\nabla v_k^{(0)},\sum_{i=1}^{{\bf L}}X_j^{(i)}\ra.
\]
Next we set
\begin{equation}\label{G2}
\begin{aligned}
\Gamma_{jk}^{(2)}:=&\sum_{n=1}^{N_2}E_{n,j,k}\la \Phi_n,\nabla\gamma_1+\nabla\gamma_2\ra+\sum_{n=1}^{N_3}F_{n,j,k}\la\Phi_n, \gamma_{1}-\gamma_{2}\ra
\end{aligned}
\end{equation}
with $E_{n,j,k}\in\R^{1\times d}$ and $F_{n,j,k}\in\R$. Finally, we define
\[
\Gamma:=\sum_{1\le j\le J}\sum_{1\le k\le J}(1+\lambda_j)^{-r/2}(\Gamma_{jk}^{(1)}+\Gamma_{jk}^{(2)})b_{jk}^{(r)}.
\]

We denote by $\mathcal{NO}^{{\mathbf L}, \bfd, \sigma}_{N, J, \phi, \Phi}$ the class of neural operators defined as above, associated with the depth ${\mathbf L}$, the width $\bfd = (d_1,...,d_{\bf L})$, the ranks $N_1, N_2, N_3, J$, the activation function $\sigma$, and the ONBs $(\{\phi_j\}, \{\Phi_n\})$. The architecture of such neural operators is illustrated in Figure~\ref{fig1} below. 
\end{definition}

\begin{remark}\label{remaa}
\begin{enumerate}
\item[(i)] The depth of neural operator ${\mathbf L}$ is different from the number of iteration steps $L$ in Proposition \ref{prop1}. But we can estimate ${\mathbf L}$ by $L$ if an appropriate activation function is used, for example, the rectified quadratic unit (ReQU), see \cite{li2019better}.
\item[(ii)] The information of $v_j^{(0)}$, $\gamma_1, \gamma_2$ and their derivatives is carried along in the hidden layer implicitly. 
\item[(iii)] In theory, the parameters ${\mathbf L}, \bfd, W^{(\ell)}, c_{nm}^{(\ell)}, b^{(\ell)}, D_j^{(\ell)}, E_{n,j,k}, F_{n,j,k}$ are all learnable; while, ranks $N_1, N_2, N_3, J$ can be a priori determined by the desired accuracy, see \eqref{ntd}, \eqref{IJ}, and \eqref{J} (with an appropriately chosen $\nu$), for example, \eqref{JJ} when $\nu=2$. 
\item[(iv)] Since $\Gamma_{jk}^{(1)}$ and also $\Gamma^{(2)}_{jk}$ can be computed independently for different $j,k$'s, in practice, the computations of $\Gamma^{(1)}_{jk}$, $\Gamma_{jk}^{(2)}$ can be done in parallel in $j,k$.   
\item[(v)] The computation of $\Gamma^{(1)}_{jk}$ involves only one integration. It can be evaluated without using a neural network approximator.  
\end{enumerate}
\end{remark}

\begin{figure}[ht]
\begin{center}
\begin{tikzpicture}[node distance=0.9cm]
\tikzstyle{inout}=[circle, draw, minimum size=0.1cm]
\tikzstyle{block} = [rectangle, draw, rounded corners, minimum width=0.8cm, minimum height=0.8cm, align=center]
\tikzstyle{arrow} = [thick,->]
\tikzstyle{label} = [midway, fill=white, above=1.5pt]
\node (I1) [draw, circle, fill=yellow!20, minimum size=0.6cm] {{$\gamma$}};
%\node (I1) [block, fill=yellow!20] {\scalebox{0.8}{$\gamma$}};
\node (I2) [block, fill=blue!20, right=of I1]{\scalebox{0.8}{$U^{(0)}$}};
 \draw[decorate, decoration={brace, mirror, amplitude=8pt}] 
        (I2.south west) -- (I2.south east) 
        node[block, fill=blue!10, midway, below=10pt] {\scalebox{0.7}{Lifting}};
\draw[->] (I1.east) -- (I2.west); 
\node[block, fill=red!10, right=of I2] (H1) {\scalebox{0.8}{$\sigma(W^{(1)} U^{(0)}(x) + (K^{(1)}_{N}U^{(0)})+b^{(1)})$}}; 
\node[block, fill=red!10, above=of H1] (DH1) {\scalebox{0.8}{$D_j^{(1)}U_j^{(1)}$}};
\draw[->, rotate=90] (H1.north) -- (DH1.south);
 \draw[->] (I2.east) -- (H1.west);
\node[block, fill=red!10, right=of H1] (HH1) {\scalebox{0.8}{$\sigma(W^{({\bf L})} U^{({\bf L})}(x) + (K^{({\bf L})}_{N}U^{({\bf L})})+b^{({\bf L})})$}};
 \draw[dash pattern=on 2pt off 3pt, thick, ->] (H1.east)--(HH1.west);% node[label]{\scalebox{0.9}{$\sigma$}};
\node[block, fill=red!10, above=of HH1] (DH2) {\scalebox{0.8}{$\sum_{\ell=1}^{\bf L}D_j^{(\ell)}U_j^{(\ell)}$}}; 
\draw[->, rotate=90] (HH1.north) -- (DH2.south);
\draw[dash pattern=on 2pt off 3pt, thick, ->] (DH1.east) -- (DH2.west); 
 \draw[decorate, decoration={brace, mirror, amplitude=15pt}, yshift=10pt, yshift=-30pt] 
        (H1.south west) -- (HH1.south east) 
        node[block, fill=red!10, midway, below=18pt] {\scalebox{0.7}{${\bf L}$'s Hidden Layers}};
 \draw[dash pattern=on 6pt off 5pt, rounded corners=15pt] (2.8,-2) rectangle (14.8,2.3);       
\node (O1) [draw, circle, fill=yellow!20] at (15.8,2.0) {\scalebox{0.65}{$v_k^{(0)}$}};
\node (O2) [block, fill=green!10, right=of HH1]{\scalebox{0.8}{$\Gamma$}};
\draw[->] (DH2.east) -- (O2.west); 
\draw[->] (O1.south) -- (O2.north);
 \draw[decorate, decoration={brace, mirror, amplitude=15pt}, yshift=10pt, yshift=-30pt] 
        (O2.south west) -- (O2.south east) 
        node[block, fill=green!10, midway, below=18pt] {\scalebox{0.7}{Projection}};
\end{tikzpicture}
\end{center}
\caption{Neural Operator Architecture}
\label{fig1}
\end{figure}
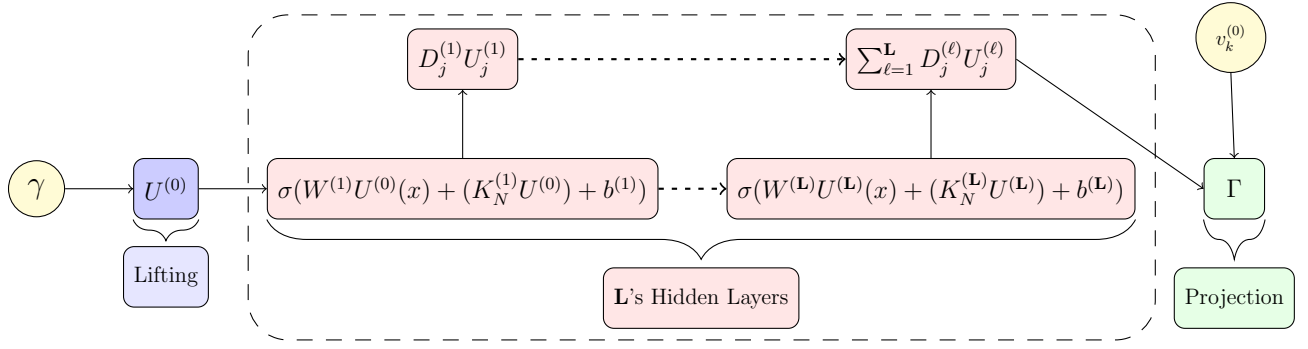

To facilitate the understanding of the hidden layers outlined above, we will describe first several layers in detail. We will use the ReQU activation function, which is defined by
\[
\sigma_2(t):=\left\{
\begin{aligned}
&t^2,\quad t\ge 0,\\
&0,\quad t<0.
\end{aligned}\right.
\]
Recall from \cite[Lemma~2.1]{li2019better}, for any $t,s\in\R$, we have that
\begin{equation}\label{q1}
t=\beta_1^T\sigma_2(\omega_1 t+\alpha_1)
\end{equation}
and
\begin{equation}\label{q2}
ts=\beta_1^T\sigma_2(\omega_1t+\alpha_1s),
\end{equation}
where
\[
\beta_1=\frac 14\begin{pmatrix}1\\1\\-1\\-1\end{pmatrix},\quad\omega_1=\begin{pmatrix}1\\-1\\1\\-1\end{pmatrix},\quad\alpha_1=\begin{pmatrix}1\\-1\\-1\\1\end{pmatrix}.
\]

Now, given another $z\in\R$, we aim to compute the multiplication $zts$ exactly by a shallow neural network. We denote $V^{(0)}=\begin{pmatrix}z&t&s\end{pmatrix}^T$ and consider $V^{(0)}$ as an input. The first hidden layer is defined by
\[
V^{(1)}=\sigma_2(A^{(1)}V^{(0)}+c^{(1)})
\]
with
\[
A^{(1)}=\begin{pmatrix}\omega_1&{\bf 0}_{4\times 1}&{\bf 0}_{4\times 1}\\{\bf 0}_{4\times 1}&\omega_1&\alpha_1\end{pmatrix}\in\R^{8\times 3},\quad c^{(1)}=\begin{pmatrix}\alpha_1\\{\bf 0}_{4\times 1}\end{pmatrix}\in\R^{8\times 1},  
\]
where for $a,b\in\N$, ${\bf 0}_{a\times b}$ denotes $a\times b$ zero matrix. The second hidden layer is given by
\[
V^{(2)}=\sigma_2(A^{(2)}V^{(1)})
\]
with 
\[
A^{(2)}=\begin{pmatrix}\omega_1&\alpha_1\end{pmatrix}\begin{pmatrix}\beta_1^T&{\bf 0}_{1\times 4}\\{\bf 0}_{1\times 4}&\beta_1^T\end{pmatrix}\in\R^{4\times 8}.
\]
Combining \eqref{q1} and \eqref{q2}, let $V^{(3)}=\beta_1^TV^{(2)}$ be the output layer, then 
\begin{equation}\label{q3}
zts=V^{(3)}=\beta_1^T\sigma_2(A^{(2)}\sigma_2(A^{(1)}V^{(0)}+c^{(1)})).
\end{equation}
In other words, we can use a neural network of two hidden layers, and one output layer to exactly represent $zts$. We only need $8$ neurons in the first hidden layer and $4$ neurons in the second hidden layer, see Figure~\ref{fig2}. 
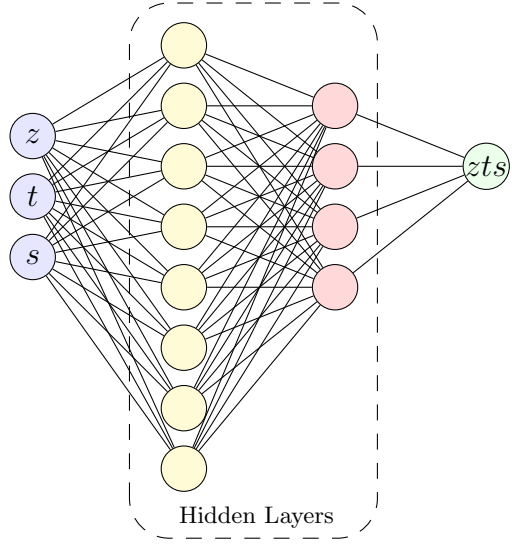
\begin{figure}[ht]
\begin{center}
    \begin{tikzpicture}[
        roundnode/.style={circle, draw=black, minimum size=6mm, inner sep=0pt, align=center},
        scale=0.8
    ]
        % Input Layer (centered)
            \node[roundnode, fill=blue!10] (I1) at (0,-2) {$z$};
            \node[roundnode, fill=blue!10] (I2) at (0,-3) {$t$};
            \node[roundnode, fill=blue!10] (I3) at (0,-4) {$s$};

        % First Hidden Layer
        \foreach \y in {1,2,3,4,5,6,7,8} {
            \node[roundnode, fill=yellow!20] (H1\y) at (2.5,-\y+0.5) {};
        }

        % Second Hidden Layer
        \foreach \y in {1,2,3,4} {
            \node[roundnode, fill=red!15] (H2\y) at (5,-\y-0.5) {};
        }

        % Output Layer
        \node[roundnode, fill=green!8] (O) at (7.5,-2.5) {$zts$};

        % Connections
        \foreach \i in {1,2,3} {
            \foreach \h in {1,2,3,4,5,6,7,8} {
                \draw (I\i) -- (H1\h);
            }
        }
        
        \foreach \h in {1,2,3,4,5,6,7,8} {
            \foreach \hnext in {1,2,3,4} {
                \draw (H1\h) -- (H2\hnext);
            }
        }
        
        \foreach \h in {1,2,3,4} {
            \draw (H2\h) -- (O);
        }
         \draw[dash pattern=on 6pt off 5pt, rounded corners=15pt] (1.6,-8.65) rectangle (5.7,0.2); 
          \node at (3.7,-8.3) {\scriptsize{Hidden Layers}};
    \end{tikzpicture}
\end{center}
\caption{Illustration of neural network for representing a polynomial of three variables}
\label{fig2}
\end{figure}

In view of \eqref{q1}, \eqref{q2}, \eqref{q3}, the first couple of hidden layers in the neural operator using ReQU given in Definition~\ref{def:neural-operator} could be designed as follows:

\medskip\noindent (H1): $U_j^{(1)}=\sigma_2(W^{(1)}U_j^{(0)}+K_N^{(1)}U_j^{(0)}+b^{(1)})$ with
\[
W^{(1)}=\begin{pmatrix}\omega_1&{\bf 0}_{4\times 1}&{\bf 0}_{4\times 1}\\{\bf 0}_{4\times 1}&\omega_1&\alpha_1\end{pmatrix},\; K_N^{(1)}=0,\; b^{(1)}=\begin{pmatrix}\alpha_1\\{\bf 0}_{4\times 1}\end{pmatrix},
\]
and also $D_j^{(1)}={\bf 0}_{1\times 8}$. Here $K_N^{(0)}=0$ means the null operator. 
\begin{comment}
\[
W^{(1)}=\begin{pmatrix}\omega_1&{\bf 0}&{\bf 0}\\{\bf 0}&\omega_1&\alpha_1\\{\bf 0}&\omega_1&{\bf 0}\\{\bf 0}&{\bf 0}&\omega_1\end{pmatrix},\; K_N^{(0)}=\text{\large{0}},\; b^{(1)}=\begin{pmatrix}\alpha_1\\{\bf 0}\\\alpha_1\\\alpha_1\end{pmatrix},
\]
and also $D_j^{(1)}=\text{\large{0}}$. Here $K_N^{(0)}=0$ stands for the null operator. 
\end{comment}

\medskip\noindent (H2): $U_j^{(2)}=\sigma_2(W^{(2)}U_j^{(1)}+K_N^{(2)}U_j^{(1)}+b^{(2)})$ with 
\[
W^{(2)}=\begin{pmatrix}\omega_1\beta_1^T&\alpha_1\beta_1^T\\{\bf 0}_{4\times 4}&\omega_1\beta_1^T\end{pmatrix}\in\R^{8\times 8},\; K_N^{(2)}=0,\; b^{(2)}=\begin{pmatrix}{\bf 0}_{4\times 1}\\\alpha_1\end{pmatrix},\; D_j^{(2)}={\bf 0}_{1\times 8}.
\]

\medskip\noindent (H3): $U_j^{(3)}=\sigma_2(W^{(3)}U_j^{(2)}+K_N^{(3)}U_j^{(2)}+b^{(3)})$, where
\[
K^{(3)}_{N}U_j^{(2)}(x) =
\sum_{1\le m,n\le N}c_{nm}^{(3)}\langle \Phi_{m}, U_j^{(2)} \rangle \Phi_{n}(x) 
\]
with 
\[
W^{(3)}=\begin{pmatrix}{\bf 0}_{4\times 1}\\\omega_1\end{pmatrix}\begin{pmatrix}{\bf 0}_{1\times 4}&\beta_1^T\end{pmatrix}=\begin{pmatrix}{\bf 0}_{4\times 4}&{\bf 0}_{4\times 4}\\{\bf 0}_{4\times 4}&\omega_1\beta_1^T\end{pmatrix}\in\R^{8\times 8},
\]
\[
c_{nm}^{(3)}=\omega_1\tilde c_{nm}^{(3)}\begin{pmatrix} \beta_1^T&{\bf 0}_{1\times 4}\end{pmatrix}\in\R^{4\times 8}\quad\mbox{for some learnable parameter}\quad\tilde c_{nm}^{(3)}\in\R,
\]
and
\[
b^{(3)}=\begin{pmatrix}\alpha_1\\\alpha_1\end{pmatrix}.
\]
It is helpful to point out that 
\[
U_j^{(3)}\approx\sigma_2\begin{pmatrix}\omega_1\left(\Psi_{q,N}^{[1]}(v_j^{(0)})\right)+\alpha_1\\\omega_1(\gamma_2\Delta\gamma_1)+\alpha_1\end{pmatrix}.
\]
Furthermore, we have
\[
\sum_{\ell=1}^3X_j^{(\ell)}=D_j^{(3)}U_j^{(3)}
\]
with
\[
D_j^{(3)}=\begin{pmatrix}\beta_1^T&{\bf 0}_{1\times 4}\end{pmatrix}\in\R^{1\times 8}.
\]

\medskip\noindent (H4): $U_j^{(4)}=\sigma_2(W^{(4)}U_j^{(3)}+K_N^{(4)}U_j^{(3)}+b^{(4)})$ with
\[
W^{(4)}=\begin{pmatrix}\omega_1\beta_1^T&\alpha_1\beta_1^T\\{\bf 0}_{4\times 4}&\omega_1\beta_1^T\end{pmatrix},\quad b^{(4)}=\begin{pmatrix}{\bf 0}_{4\times 1}\\\alpha_1\end{pmatrix},\quad K_N^{(4)}=0,
\]
and $D_j^{(4)}={\bf 0}_{1\times 8}$. In other words, (H4) has the same structure as (H2). 

Finally, it is not hard to see that (H5) has the same structure as (H3). Consequently, the rest of hidden layers could be designed alternate between (H2) and (H4). In view of the above construction, it follows that the iteration step $L$ in Proposition~\ref{prop1} and the depth of the layer ${\bf L}$ is related by ${\bf L}\approx 2L$ if we use ReQU as the activation function. The width (number of nodes) at each hidden layer is at most $8$.  We are now ready to prove the main result of the paper. 

\begin{theorem}{\rm[Quantitative Approximation Theorem]}\label{thm:quantitative-UAT}
Assume that $d=2,3$, $r>0$, and $s>2+d/2$, and fix
$\eta_0\in(0,1)$.
For every $\eps\in(0,1)$ there exist a depth ${\bf L}$,
a width $\bfd$, ranks $N=(N_1,N_2,N_3)$ and $J$, and
${\mathcal G}\in\mathcal{NO}^{{\mathbf L},\bfd,\sigma_2}_{N,J,\phi,\Phi}$
such that
\begin{equation}\label{neuralapp}
\sup_{\gamma\in{\mathcal S}^{M,\eta_0}_{m,B'}}
\|\tilde\Lambda_\gamma-{\mathcal G}(\gamma)\|_{\H_r}
\le\eps.
\end{equation}
Moreover,
\begin{equation}\label{LG}
{\bf L}({\mathcal G})
\le C_1\bigl(1+\log(\eps^{-1})\bigr)
\end{equation}
and
\begin{equation}\label{dG}
d_\ell({\mathcal G})\le C_2.
\end{equation}
Here $C_1$ depends only on $d,B,B',s,r,m,M,\eta_0$,
and $C_2$ is an absolute constant (the channel width per branch).
For any fixed $t\in(0,2-d/2)$, the ranks can be chosen with
\[
N_1=O(\eps^{-2d/t}),\qquad
N_2=O(\eps^{-\kappa_1}),\qquad
N_3=O(\eps^{-\kappa_2}),\qquad
J=O(\eps^{-1/2}),
\]
where $\kappa_1,\kappa_2$ are those in \eqref{IJ}.
The implied constants are uniform in
$\gamma\in{\mathcal S}^{M,\eta_0}_{m,B'}$ and may depend on
the fixed a priori parameters and on $t$.
\end{theorem}
\begin{proof}
All conductivities in ${\mathcal S}^{M,\eta_0}_{m,B'}$ have the
same lower bound $m$, a common $H^s(B)$ bound, and hence a
common $W^{2,\infty}(B)$ bound.
Moreover, the contraction constants of both the exact and
truncated Picard operators are bounded above by $\eta_0<1$.
Thus the constants in \eqref{error1}, Proposition~\ref{prop0},
and Proposition~\ref{prop1} can be chosen uniformly on this class.\\
Apply these approximation estimates with an internal tolerance
$\tau\in(0,1)$, using the same ranks and iteration depth for all
conductivities in ${\mathcal S}^{M,\eta_0}_{m,B'}$.
Combining \eqref{dncoeff}, \eqref{dnmc}, \eqref{err0}, and
\eqref{psin} with \eqref{error1}, and using the ReQU realization
described above, gives an operator $\mathcal G_\tau$ with
\[
\sup_{\gamma\in{\mathcal S}^{M,\eta_0}_{m,B'}}
\|\tilde\Lambda_\gamma-\mathcal G_\tau(\gamma)\|_{\H_r}
\le C_*\tau,
\]
where $C_*\ge1$ depends only on the fixed a priori parameters.
Choose $\tau=\eps/C_*$.
This proves \eqref{neuralapp}, without changing the powers of
$\eps$ in the rank estimates.\\
By \eqref{depth},
\[
L=\left\lceil
\frac{2\log(\tau^{-1})}{\log(\eta_0^{-1})}
\right\rceil.
\]
The ReQU realization uses at most $2L+1$ hidden layers, giving
\eqref{LG}. The channel width per branch is bounded independently
of $\eps$, as in the construction above.
\end{proof}

\begin{remark}\label{rem6}
As already pointed out in (iii) of Remark~\ref{remaa}, ranks $N_1, N_2, J$ are a priori determined by $\eps$ and $J$ behaves like $\eps^{-1/2}$. To achieve the desired accuracy $\eps$, the depth of neural operators grows only \emph{logarithmically} in $\eps^{-1}$. Furthermore, the architecture of neural operators given in Definition~\ref{def:neural-operator} satisfies the discretization-invariance property, which is explained in detail in Appendix~\ref{A2}.  
\end{remark}

\section{Application to the inverse problem}

In this section, we describe how to use the neural operator $\mathcal{G}$, constructed above, to solve the inverse problem of recovering the conductivity function $\gamma$ from the normalized DN map $\tilde{\Lambda}_\gamma$. This inverse boundary value problem, also known as \emph{Calderón's problem} or \emph{electrical impedance tomography}, has been extensively studied. Fundamental questions such as uniqueness, stability, and reconstruction methods have been addressed; see, for example, \cite{feldman2019calderon}. 

Here, we are interested in reconstructing the conductivity function $\gamma$ from  $\tilde{\Lambda}_\gamma$ using the Ensemble Kalman Inversion (EnKI) method \cite{iglesias2013ensemble, schillings2017analysis, schillings2018convergence}. EnKI is a derivative-free Bayesian inference algorithm that iteratively updates an ensemble of parameter samples based on observational data. In principle, each prediction step requires evaluating the forward map $\gamma \mapsto \tilde{\Lambda}_\gamma$, which can be computationally expensive due to the need to solve a PDE for each ensemble member. To address this issue, we replace the exact forward operator with a learned neural operator surrogate $\mathcal{G}$ to reduce computational cost. The numerical results in Section~\ref{simulation} suggest that this substitution maintains reconstruction accuracy in the examples considered. However, we do not establish a theoretical bound on how the forward approximation error propagates through EnKI to the reconstructed conductivity. For applications of neural networks to MCMC methods for Bayesian inverse problems, see, for example, \cite{majee2025mcmc,yan2017convergence}.

We now provide a detailed description of the EnKI method. To begin, we consider a lower rank approximator of $\tilde\Lambda_\gamma$, i.e.,
\[
\pi_{KK}\tilde\Lambda_\gamma=\sum_{j,k=1}^Kc_{jk}^{(r)}(\gamma)b_{jk}^{(r)}
\]
for some given $K\in\N$, see \eqref{pjj}. Accordingly, we represent the truncated map $\pi_{KK} \tilde{\Lambda}_\gamma$ by the $K \times K$ coefficient matrix $\big(c_{jk}^{(r)}(\gamma)\big)$. We denote its surrogate, constructed using a neural operator, by ${\mathcal{G}}(\gamma)$. The inverse problem is then formulated as a Bayesian inference problem with noisy observations given by
\begin{equation} \label{data}
y = \G(\gamma) + \eta,
\end{equation}
where $\eta \sim \mathcal{N}(0, \Xi)$ is a mean-zero Gaussian random variable with known covariance $\Xi$ defined on the observation space $Y:=\R^{K\times K}$. For simplicity, we consider a smooth conductivity $\gamma$ in what follows.  

In order to apply $\mathcal{G}(\gamma)$ at each iteration of the EnKI algorithm, it is necessary to ensure that the conductivity field $\gamma$ is strictly positive and that the perturbation $\gamma - 1$ is compactly supported within the domain $B$. To achieve this, we define $\mu$ to be a Gaussian random field on $\mathbb{R}^d$ with a suitable Whittle-Mat\'ern covariance kernel, and introduce a smooth cut-off function $\chi(x)$ satisfying $\chi(x) = 1$ for $x \in B'$ and $\chi(x) = 0$ for $x \in \mathbb{R}^d \setminus B''$, where $B' \subset B'' \subset B$ are nested subdomains. This construction ensures that the random field is spatially localized and the resulting conductivity field satisfies the required positivity and support conditions. We then define the conductivity field by
\begin{equation}\label{link}
\gamma(x) = \exp\left( \chi(x) \sigma(x) \right)
\end{equation}
with $\sigma\sim \mu$, which guarantees that $\gamma(x) > 0$ for all $x \in {\mathbb R}^d$ and that $\gamma(x) = 1$ outside of $B''$.

Let $\sigma\sim\mu$ and define $\tilde\G(\sigma)=\G(\exp(\chi\sigma))$. We next define ${\mathcal{F}}: Z \to Z$ by
\[
\mathcal{F}(z) = 
\begin{pmatrix}
\sigma \\
\tilde{\G}(\sigma)
\end{pmatrix}
\quad \text{for} \quad 
z = 
\begin{pmatrix}
\sigma \\
y
\end{pmatrix}
\in Z,
\]
where $\sigma\in C^{\infty}(\R^d)$ and $y\in\R^{K\times K}$, and the corresponding artificial dynamical system 
\begin{equation}\label{dynamics}
z_{k+1}={\mathcal F}(z_k).
\end{equation}
Related to the artificial dynamical, \eqref{data} can be also expressed as
\begin{equation}\label{data1}
y_{k+1}={\mathcal H}z_{k+1}+\eta_{k+1},
\end{equation}
where the projection operator ${\mathcal H}:Z\to Y$ is defined by ${\mathcal H}=(0,I)$ and $\{\eta_k\}$ is an iid sequence with $\eta_1\sim N(0,\Gamma)$ and $\Gamma$ is given above. 

The EnKI algorithm begins with an initial (or first guess) ensemble of particles $\{z_0^{(j)}\}_{j=1}^J$ which will be iteratively updated with the ensemble Kalman filter method. The initial ensemble $\{z_0^{(j)}\}_{j=1}^J$ can be constructed by choosing an ensemble $\{\sigma_0^{(j)}\}_{j=1}^J$ with $\sigma_0^{(j)}\overset{\text{\scriptsize iid}}{\sim}\mu$ and setting
\[
z_0^{(j)}=\begin{pmatrix}\sigma_0^{(j)}\\\tilde\G(\sigma_0^{(j)})\end{pmatrix},\quad j=1,\cdots,J.
\]
We now describe the iterative ensemble Kalman filter method for \eqref{dynamics}, \eqref{data1}.

\smallskip
\noindent{\bf Algorithm}. Let $\{z_0^{(j)}\}_{j=1}^J$ be the initial ensemble. For $k=0,1,2,\cdots$
\begin{enumerate}
\item[(i)]{\it Prediction step}. We first propagate particles under the dynamics \eqref{dynamics}:
\[
\widetilde{z}^{(j)}_{k+1}=\begin{pmatrix}\widetilde{\sigma}^{(j)}_{k+1}\\\widetilde y^{(j)}_{k+1}\end{pmatrix}={\mathcal F}(z_{k}^{(j)})=\begin{pmatrix}{\sigma}^{(j)}_{k}\\{\tilde\G}({\sigma}^{(j)}_{k})\end{pmatrix},\quad j=1,\cdots,J.
\]
Compute the empirical mean and covariance:
\[
\begin{aligned}
&\widehat z_{k+1}=\frac 1J\sum_{j=1}^J\widetilde{z}^{(j)}_{k+1}\\
&{\mathcal C}_{k+1}=\frac 1J\sum_{j=1}^J\widetilde{z}^{(j)}_{k+1}(\widetilde{z}^{(j)}_{k+1})^T-\widehat z_{k+1}(\widehat z_{k+1})^T.
\end{aligned}
\]
\item[(ii)]{\it Analysis step}. Define the \emph{Kalman gain} ${\mathcal K}_{k+1}$ by
\[
{\mathcal K}_{k+1}={\mathcal C}_{k+1}{\mathcal H}^T({\mathcal H}{\mathcal C}_{k+1}{\mathcal H}^T+\Gamma)^{-1}.
\]
Update each ensemble particle by
\[
z_{k+1}^{(j)}=\widetilde{z}^{(j)}_{k+1}+{\mathcal K}_{k+1}(y_{k+1}^{(j)}-{\mathcal H}\widetilde{z}^{(j)}_{k+1})=(I-{\mathcal K}_{k+1}{\mathcal H})\widetilde{z}^{(j)}_{k+1}+{\mathcal K}_{k+1}y_{k+1}^{(j)},
\]
where $y_{k+1}^{(j)}=y+\eta^{(j)}_{k+1}$ with $\eta_{k+1}^{(j)}\sim N(0,\Gamma)$. Here $y$ is the observation data and $y_{k+1}^{(j)}$ is obtained by artificially perturbing the observation data, see \eqref{data1}.

\item[(iii)] Finally, we compute the mean of the parameter update
\[
\bar \sigma_{k+1}=\frac 1J\sum_{j=1}^J{\mathcal H}^\perp z_{k+1}^{(j)}=\frac 1J\sum_{j=1}^J \sigma_{k+1}^{(j)},
\]
where ${\mathcal H}^\perp=[I,0]$. Then $\gamma_{k+1}=\exp(\bar\sigma_{k+1})$ is a refined estimate of the unknown conductivity $\gamma$. 
\end{enumerate}
It is important to point out that  for every $(k,j)\in{\mathbb N}\times\{1,2,\cdots,J\}$ we have $\sigma_{k+1}^{(j)}\in{\mathcal A}$ and hence $\bar \sigma_{k+1}\in{\mathcal A}$ for all $k\in{\mathbb N}$, where ${\mathcal A}=\text{span}\{\sigma_0^{(j)}, j=1,\cdots,J\}$, see \cite[Theorem~2.1]{iglesias2013ensemble}. 

\section{Numerical simulations}\label{simulation}

\subsection{Simulation of the DN map}

We generated a total of 4,000 samples from the DN map, allocating 3,200 samples for training and 800 for testing. The modeling task focuses on predicting the projection coefficients $c_{jk}^{(0)}$, i.e., the coefficients of the low-rank approximation introduced in \eqref{pjj} with $r=0$, based on four input variables denoted by $U^{(0)}_j$, as described in Section~\ref{sec3} (Lifting). As a preliminary check for dataset consistency, we examined the statistical distribution of the output coefficients. The training set has a mean of 353.43 and a standard deviation of 45.36, while the testing set shows a closely aligned mean of 351.90 and standard deviation of 46.75. This statistical agreement ensures consistent training and evaluation conditions, which is essential for reliable generalization and model robustness.

For model development, we performed empirical tuning of key hyperparameters using a trial-and-error approach. Since the learning rate is one of the most critical hyperparameters that directly influences both the speed and stability of convergence, it must be carefully selected to ensure effective training. A value that is too high may cause the optimization to diverge, whereas a value that is too low can result in excessively slow learning or entrapment in local minima. The Adam optimizer was selected for its adaptive learning rate mechanism and momentum-based updates, which enable faster and more stable convergence compared to traditional stochastic gradient descent. To evaluate its impact, the network was trained with five different learning rates:
: 10\textsuperscript{-4}, 10\textsuperscript{-3}, 
5*10\textsuperscript{-3}, 10\textsuperscript{-2}, and 10\textsuperscript{-1}. As shown in Figure~\ref{learning}(a), a learning rate of 0.01 achieved the lowest Mean Absolute Percentage Error (MAPE) of 11.11\%. This rate offered a good balance between convergence speed and stability. In contrast, an intermediate learning rate (0.01) led to suboptimal convergence, while larger rates (0.1) induced training instability and overshooting.

\begin{figure}
    \centering
    \includegraphics[width=1\linewidth]{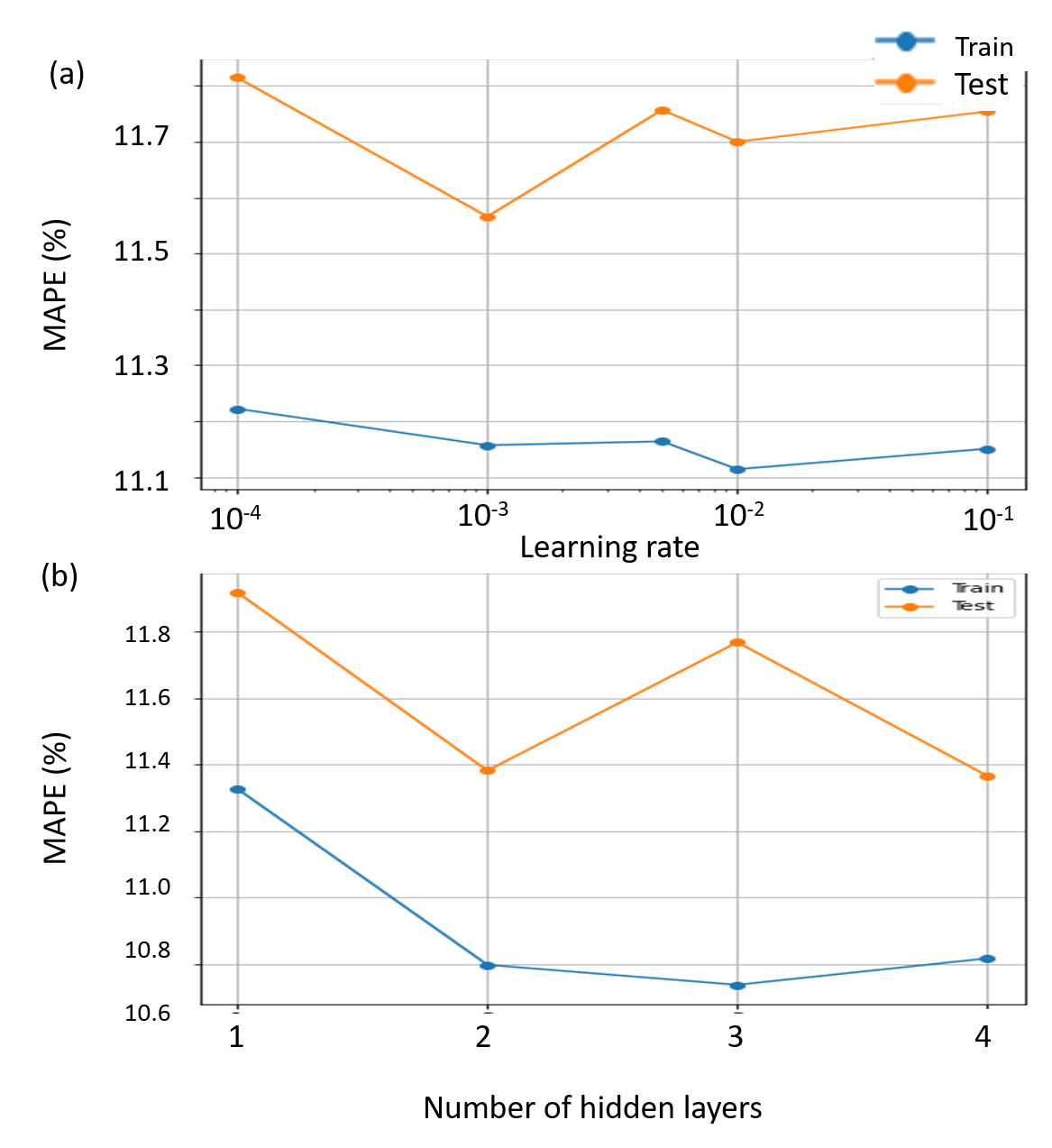}
    \caption{Model performance comparison with (a) different learning rates and (b) varying numbers of hidden layers.}
    \label{learning}
\end{figure}

After analyzing the impact of learning rate, we further examined another key architectural factor—the model depth. Optimizing the number of hidden layers is crucial because it directly affects the model’s representational capacity and its ability to capture nonlinear patterns in data. To investigate this effect, we tested architectures with one, two, three, and four hidden layers, as shown in Figure~\ref{learning}(b). Interestingly, the three-hidden-layer model again achieved the lowest MAPE (10.54\%), outperforming deeper architectures. This indicates that while increasing depth from one or two hidden layers to three improves the model's capacity to capture the underlying patterns in the data, further increasing to four hidden layers results in diminished performance. The performance drop with four layers suggests that excessive depth may introduce overfitting, particularly given the limited size of our dataset (4000 samples). This finding aligns with the principle of Occam’s Razor in machine learning, which favors simpler models that generalize well when they sufficiently explain the data.

These experiments underscore the importance of deliberate hyper-parameter selection and architecture design in physics-informed machine learning. In our case, incorporating domain knowledge, such as the radial symmetry in voltage/current distributions and smoothness of the projection map, helped avoid overfitting and improved training efficiency. Future improvements may involve incorporating regularization or physics-informed loss functions to further enhance generalization, particularly in more complex geometries or under higher-resolution settings.

To gain further insight into data structure and boundary effects, we analyzed six representative samples under boundary conditions defined by $\sin\left(\frac{k\theta}{\sqrt{\pi}}\right)$ and $\cos\left(\frac{k\theta}{\sqrt{\pi}}\right)$ for $k = 1, 2, 3$. Figures~\ref{fig:sin_bcs} and~\ref{fig:cos_bcs} display corresponding voltage and current distributions. Each subfigure illustrates how increasing the harmonic index $k$ influences spatial profiles. Notably, voltage distributions exhibit strong radial symmetry, a consequence of the circular domain. Current profiles, shaped by Neumann boundary conditions, also exhibit symmetry but with larger variations near the boundary. For higher-frequency inputs (e.g., $k=3$), steeper voltage gradients cause increased current magnitudes near the edges, evident in Figures~\ref{fig:sin_bcs}(iii) and~\ref{fig:cos_bcs}(iii).

\begin{figure}
    \centering
    \includegraphics[width=0.8\linewidth]{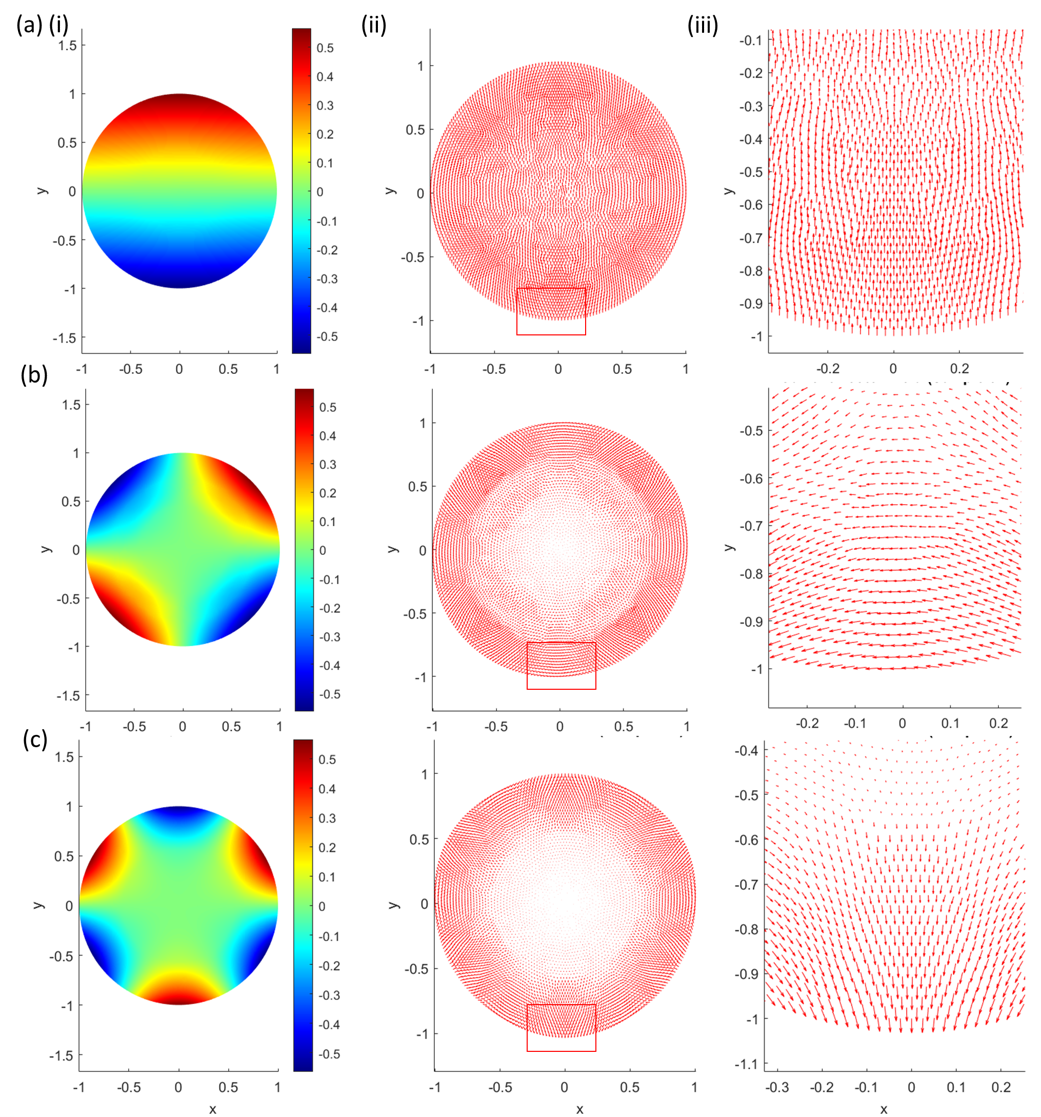}
    \caption{Voltage and current distributions with boundary conditions given by $\sin\left(\frac{k\theta}{\sqrt{\pi}}\right)$  for $k=$ (a) 1, (b) 2, and (c) 3. Subfigures: (i) voltage distributions, (ii) current distributions, and (iii) magnified view of the current distributions near the boundary.}
    \label{fig:sin_bcs}
\end{figure}

\begin{figure}
    \centering
    \includegraphics[width=0.8\linewidth]{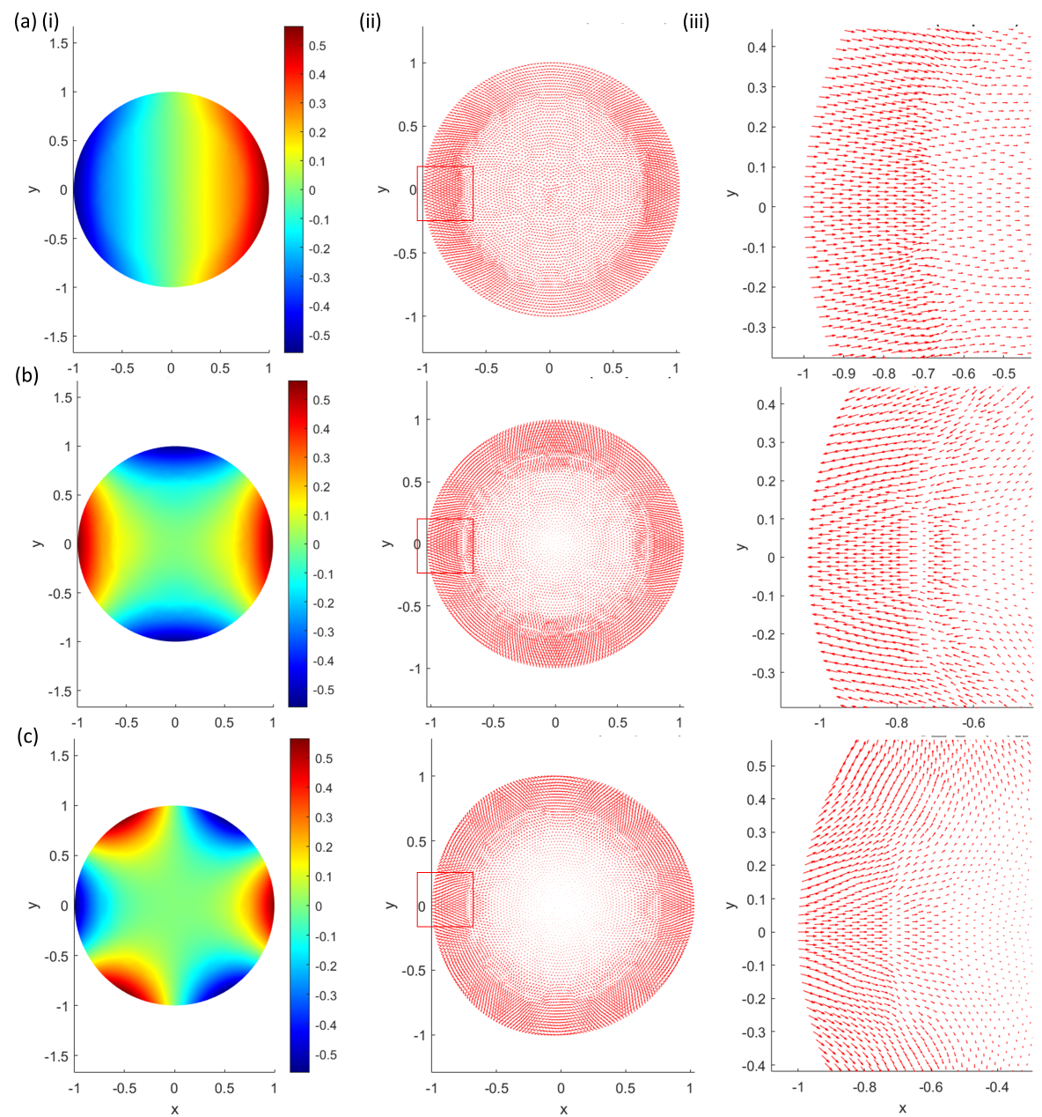}
    \caption{Voltage and current distributions with boundary conditions given by $\cos\left(\frac{k\theta}{\sqrt{\pi}}\right)$ for $k=$ (a) 1, (b) 2, and (c) 3. Subfigures: (i) voltage distributions, (ii) current distributions, and (iii) magnified view of the current distributions near the boundary.}
    \label{fig:cos_bcs}
\end{figure}

Table~\ref{tab:prediction-performance} summarizes model performance in predicting the projection coefficient~$c$. Although predictions (352.50-351.50) consistently underestimate ground truth values (281.99-437.33), the error is more manageable for lower target values (e.g., Figs.~\ref{fig:sin_bcs}(b) and~\ref{fig:cos_bcs}(c)), where the MAE ranges between 10.48 and 25.33 and the MAPE remains around 2.89-6.72\%. This consistent underestimation indicates a systematic bias in the model, which could potentially be mitigated through the incorporation of additional input features.

\begin{table}
\renewcommand{\arraystretch}{1.5}
\centering
\begin{tabular}{c|c|c|c|c}
\hline
Sample & True Value & Predicted Value & MAE & MAPE (\%) \\
\hline \hline
Fig. 4(a) & 437.33 & 352.50 & 84.83 & 19.40 \\
Fig. 5(a) & 281.99 & 352.50 & 70.51 & 25.00 \\
Fig. 4(b) & 409.43 & 351.75 & 57.68 & 14.09 \\
Fig. 5(b) & 362.23 & 351.75 & 10.48 & 2.89 \\
Fig. 4(c) & 387.73 & 351.50 & 36.23 & 9.34 \\
Fig. 5(c) & 376.83 & 351.50 & 25.33 & 6.72 \\
\hline
\end{tabular}
\caption{Model prediction performance for the projection coefficient~$c$.}
\label{tab:prediction-performance}
\end{table}

In conclusion, the model produces stable and reasonably accurate estimates for projection coefficients in the mid-value range. While the absolute and relative errors point to improvement potential, their consistent behavior suggests that targeted strategies such as feature engineering, calibration layers, or ensemble methods, could enhance accuracy without significantly increasing computational cost. Future efforts should also prioritize robustness under complex boundary variations and scalability to high-resolution geometries.

%%%%%%%%%%%%
\subsection{Simulation of the inverse problem}
%%%%%%%%%%%%%%%%%%%%%%%%%%

%%%%%%%%%%%%%%%%%%%

Let us pick a ground truth $\sigma^\dagger\sim\mu$ and the corresponding conductivity is defined by $\gamma^\dagger(x)=\exp(\chi(x)\sigma^\dagger(x))$. Figure~\ref{fig:gamma}(a) shows the value of $\gamma^\dagger(x)$. The synthetic observation data is given by
\[
y^\dagger=(c_{jk}^{(0)}(\gamma^\dagger))_{j,k=1}^K\in\R^{K\times K}
\]
obtained by the FEM. To estimate the conductivity, we employ the Ensemble Kalman Inversion (EnKI) method based on the algorithm described above. In the \emph{prediction step}, the forward map ${\mathcal G}$ (or $\tilde{\mathcal G}$) is evaluated using both the PDE toolbox and the previously constructed ML algorithm. Figure~\ref{fig:gamma} presents the reconstruction results. The reconstruction using the ML algorithm, Fig.~\ref{fig:gamma}(c), achieves a mean absolute percentage error (MAPE) of 80.82\%, which is lower than the 88.50\% obtained using the PDE toolbox, Fig.~\ref{fig:gamma}(b). The magnitude of this reduction, though modest in a single sample, is consistent with the systematic performance gap observed in Table~\ref{tab:reconstructed-performance} below. The large MAPE arises from the fact that we use only a few coefficients $c_{jk}$ (we chose $K=3$ in simulation). We believe that the MAPE can be significantly reduced by incorporating more observation data $c_{jk}$.   
\begin{figure}[ht]
    \includegraphics[width=0.84\linewidth]{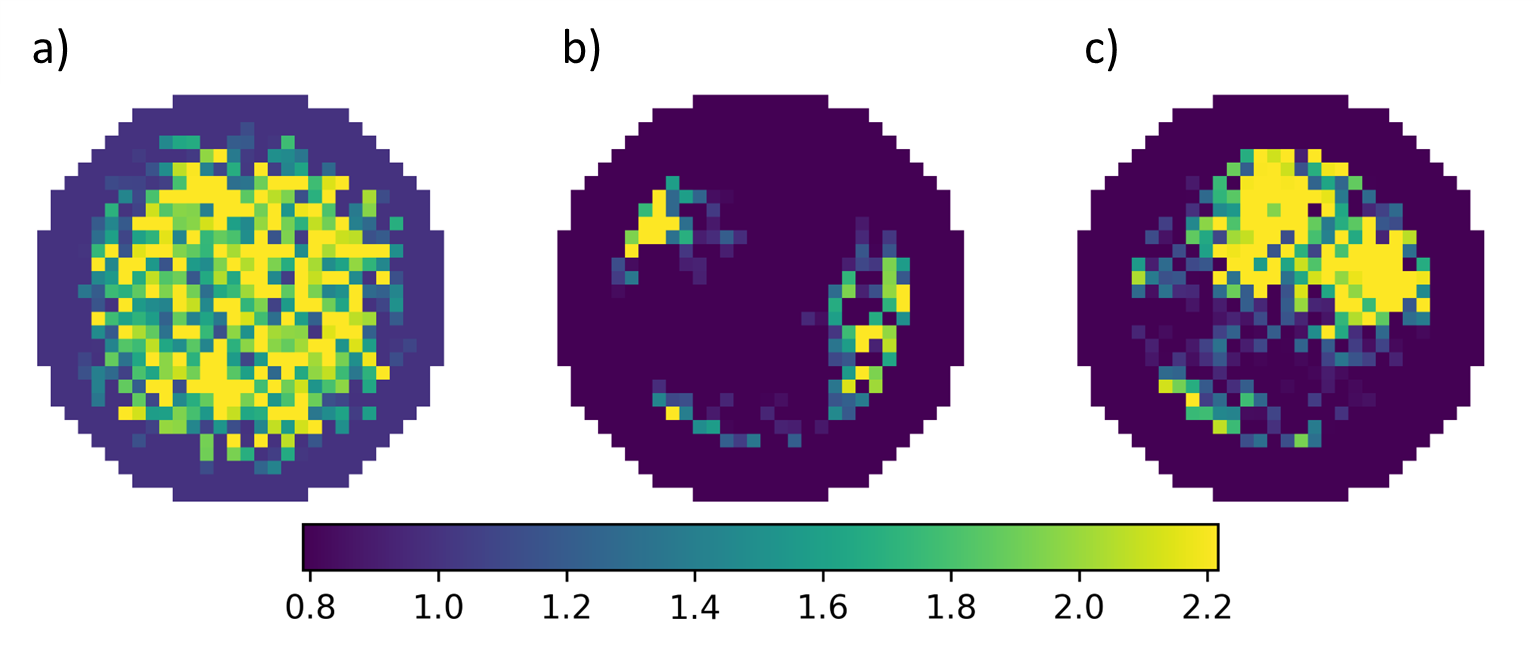}
    \centering
    \caption{EnKI method for a selected testing sample $\gamma^\dagger(x)$. (a) True coefficient; (b) Evaluation of the forward map by direct pde toolbox; (c) Evaluation of the forward map by ML.}
    \label{fig:gamma}
\end{figure}

To further demonstrate the superiority of the ML algorithm, we also apply the EnKI method to a large set of testing samples. The quantitative performance of the reconstructed coefficient fields is summarized in Table~\ref{tab:reconstructed-performance}. Our results show that reconstructions based on the ML algorithm consistently achieve lower MAE compared to those obtained using the direct PDE toolbox. This improvement is likely due to the smoother statistical properties of the ML-generated outputs, which help suppress discretization noise in the forward solution, providing more stable observational data for the inversion process. These findings highlight the robustness of the ML-assisted approach in producing reliable reconstructions and indicate that the inversion process benefits from ML-simulated forward models, which serve as higher-quality surrogates for observational data and enhance the stability of EnKI reconstructions. 
\begin{table}[ht!]
\centering
\begin{tabular}{lcccc}
\hline
 & MAE \\
\hline
Evaluation of the forward map by the direct pde toolbox & 0.87 \\
Evaluation of the forward map by the ML algorithm &0.72  \\
\hline
\end{tabular}
\caption{Performance Comparison of the EnKI Method}
\label{tab:reconstructed-performance}
\end{table}

%%%%%%%
\appendix

\section{Harmonic functions}\label{A1}

We consider eigenfunctions of the Laplace-Beltrami operator $\Delta_{\partial B}$ on the unit sphere in $\R^d$ with $d=2,3$. In $\R^3$, we have that, for $\ell\in\N\cup\{0\}$ and $m\in\Z$ satisfying $-\ell\le m\le\ell$,
\begin{equation}\label{r3}
Y_{\ell}^m(\theta,\phi)=
\left\{
\begin{aligned}
&Y_\ell^0(\theta, \phi) = N_\ell^0 P_\ell^0(\cos\theta)\quad\mbox{for}\quad m=0,\\
&Y_\ell^m(\theta, \phi) = \sqrt{2} N_\ell^m P_\ell^m(\cos\theta) \cos(m\phi)\quad\mbox{for}\quad m>0,\\
&Y_\ell^{-m}(\theta, \phi) = \sqrt{2} N_\ell^m P_\ell^m(\cos\theta) \sin(m\phi)\quad\mbox{for}\quad m<0,
\end{aligned}\right.
\end{equation}
are normalized eigenfunctions of $\Delta_{\partial B}$ corresponding to the eigenvalue $-\ell(\ell+1)$, where $0\le\theta<\pi$, $0\le\phi<2\pi$, the coefficients 
\[
N_\ell^m=\sqrt{\frac{(2\ell+1)}{4\pi}\frac{(\ell-m)!}{(\ell+m)!}}.
\]
We can flatten double indices $(\ell,m)$ into single indices $j$ by defining $j=\ell^2+(\ell+m)$, for $j\in\N\cup\{0\}$. Conversely, for any $j\in\N\cup\{0\}$, we can find a unique pair of $(\ell,m)$ with $\ell\in\N\cup\{0\}$ and $-\ell\le m\le \ell$ given by setting $\ell=\lfloor{\sqrt{j}}\rfloor$, i.e., $\ell^2\le j<(\ell+1)^2$, and $m=j-\ell^2-\ell$.

In $\R^2$, we have that for $\ell\in\N\cup\{0\}$
\begin{equation}\label{r2}
Y^m_\ell(\theta)=
\left\{
\begin{aligned}
&\frac{1}{\sqrt{2\pi}}\quad\mbox{for}\quad \ell=0,\\
&\frac{1}{\sqrt{\pi}}\cos(\ell\theta)\quad\mbox{for}\quad m=\ell>0,\\
&\frac{1}{\sqrt{\pi}}\sin(\ell\theta)\quad\mbox{for}\quad m=-\ell<0,
\end{aligned}\right.
\end{equation}
are normalized eigenfunctions of $\Delta_{\partial B}$ corresponding to the eigenvalue $-\ell^2$. Similarly, the double indices $(\ell,m)$ can be flattened into
a single index $j$ by
\[
j=
\begin{cases}
0, & \ell=0,\\
2\ell-1, & m=\ell,\quad \ell\ge1
\quad\text{(cosine terms)},\\
2\ell, & m=-\ell,\quad \ell\ge1
\quad\text{(sine terms)}.
\end{cases}
\]
Conversely, for $j\in\mathbb N\cup\{0\}$, we recover
$(\ell,m)$ as follows:
\[
(\ell,m)=
\begin{cases}
(0,0), & j=0,\\[1mm]
\left(\dfrac{j+1}{2},\dfrac{j+1}{2}\right),
& j\ge1 \text{ odd},\\[3mm]
\left(\dfrac{j}{2},-\dfrac{j}{2}\right),
& j\ge2 \text{ even}.
\end{cases}
\]
In other words, we can see that $\ell\sim j^{1/(d-1)}$ for $d=2,3$.

For simplicity, we use $\Theta$ to denote the spherical variables on $\S^1$ or $\S^2$. That is, $Y_\ell^m(\Theta)$ is given in \eqref{r3} or \eqref{r2} in $\R^3$ or $\R^2$. Therefore, defining
\begin{equation*}%\label{harmo}
\phi^{(0)}_j(x):=\frac{1}{R^{(d-1)/2}}Y_\ell^m(\Theta)
\end{equation*}
with $x=R\Theta$, where indices $j$ and $(\ell,m)$ are given above, $\{\phi_j^{(0)}: j\in\N\cup\{0\}\}$ forms an orthonormal basis of $L^2(\partial B)$. In particular, when $\partial B$ is a circle of radius $R>0$ centered at the origin, $\{\phi_k^{(0)}\}_{k\in\N\cup\{0\}}=\{\frac{1}{\sqrt{2\pi R}},\{\frac{1}{\sqrt{\pi R}}\cos (k\theta), \frac{1}{\sqrt{\pi R}}\sin (k\theta)\}_{k\in\N}\}$ forms an orthonormal basis of $L^2(\partial B)$. 

Now we consider the harmonic function in $B$ with $\phi_j^{(0)}$ on $\partial B$, i.e.,
\begin{equation}\label{harmo}
\left\{
\begin{aligned}
&\Delta w=0\quad\mbox{in}\quad B,\\
&w=\phi_j^{(0)}\quad\mbox{on}\quad\partial B.
\end{aligned}\right.
\end{equation}
Then $w$ can be written in terms of spherical coordinates (or polar coordinates)
\[
w_j(x)=w_j(\rho,\Theta):=\frac{\rho^\ell}{R^\ell}\phi_j^{(0)},\quad 0\le \rho<R.
\]
Now, since $B$ is the ball of radius $R$ centered at the origin, for Dirichlet conditions $\{\frac{1}{\sqrt{2\pi R}},\{\frac{1}{\sqrt{\pi R}}\cos (k\theta), \frac{1}{\sqrt{\pi R}}\sin (k\theta)\}_{k\in\N}\}$, the corresponding harmonic functions (in terms of polar coordinates) satisfy \eqref{harmo} are given by
\[
\left\{\frac{1}{\sqrt{2\pi R}}, \left\{\frac{1}{\sqrt{\pi R^{2k+1}}}\rho^k\cos(k\theta), \frac{1}{\sqrt{\pi R^{2k+1}}}\rho^k\sin(k\theta)\right\}_{k\in\N}\right\},\;\;0\le\rho<R,\;\;0\le\theta<2\pi.
\]

Now we would like to derive some estimates of $w_j$. We can write
\[
w_j(\rho,\Theta)
=R^{-(d-1)/2}\left(\frac{\rho}{R}\right)^\ell Y_\ell^m(\Theta).
\]
Using \(dx=\rho^{d-1}\,d\rho\,dS_\Theta\), we obtain
\[
\begin{aligned}
\|w_j\|_{L^2(B)}^2=R^{-(d-1)-2\ell}
\int_0^R \rho^{2\ell+d-1}\,d\rho
\int_{\S^{d-1}}|Y_\ell^m|^2\,dS_\Theta=\frac{R}{2\ell+d},
\end{aligned}
\]
that is,
\begin{equation}\label{estw20}
\|w_j\|_{L^2(B)}
=\left(\frac{R}{2\ell+d}\right)^{1/2}\le C_{R,d}j^{-1/(2d-2)}.
\end{equation}
For the gradient, harmonicity and Green's identity yield
\[
\|\nabla w_j\|_{L^2(B)}^2
=\int_{\partial B}w_j\,\partial_\nu w_j\,dS.
\]
Since
\[
\partial_\nu w_j=\frac{\ell}{R}w_j
\quad\text{on}\;\;\partial B,
\quad
\|w_j\|_{L^2(\partial B)}=1,
\]
we obtain
\begin{equation}\label{estw21}
\|\nabla w_j\|_{L^2(B)}
=\left(\frac{\ell}{R}\right)^{1/2}\le C_Rj^{1/(2d-2)}. 
\end{equation}
The latter identity also holds for $\ell=0$, when $w_j$ is constant.

\[
\begin{aligned}
\|w_j\|_{L^4(B)}^4=R^{-2(d-1)-4\ell}
\int_0^R \rho^{4\ell+d-1}\,d\rho
\int_{\S^{d-1}} |Y_\ell^m|^4\,dS_\Theta=\frac{R^{2-d}}{4\ell+d}
\|Y_\ell^m\|_{L^4(\S^{d-1})}^4,
\end{aligned}
\]
i.e.,
\[
\|w_j\|_{L^4(B)}
=\frac{R^{(2-d)/4}}{(4\ell+d)^{1/4}}
\|Y_\ell^m\|_{L^4(\S^{d-1})}.
\]

For the gradient, suppose \(\ell\geq1\). The polar-coordinate
formula gives
\[
\nabla w_j
=R^{-(d-1)/2-\ell}\rho^{\ell-1}
\left(\ell Y_\ell^m\,\Theta+\nabla_{\S^{d-1}}Y_\ell^m\right).
\]
Since \(\Theta\perp\nabla_{\S^{d-1}}Y_\ell^m\),
\[
|\nabla w_j|^4
=R^{-2(d-1)-4\ell}\rho^{4\ell-4}
\left(\ell^2|Y_\ell^m|^2+|\nabla_{\S^{d-1}}Y_\ell^m|^2\right)^2.
\]
Integrating yields
\[
\|\nabla w_j\|_{L^4(B)}
=
\frac{R^{-(d+2)/4}}{(4\ell+d-4)^{1/4}}
\left[
\int_{\S^{d-1}}
\left(
\ell^2|Y_\ell^m|^2
+|\nabla_{\S^{d-1}}Y_\ell^m|^2
\right)^2\,dS_\Theta
\right]^{1/4}.
\]
For \(\ell=0\), \(w_0\) is constant and \(\nabla w_0=0\).

To obtain bounds depending only on the degree, we use
\[
\begin{aligned}
\|Y_\ell^m\|_{L^4(\S^{d-1})}
&\leq C_d\ell^{\alpha_d},\\
\|\nabla_{\S^{d-1}}Y_\ell^m\|_{L^4(\S^{d-1})}
&\leq C_d\ell^{1+\alpha_d},
\end{aligned}
\qquad
\alpha_2=0,\quad \alpha_3=\frac18
\]
see \cite{sogge1988concerning} for $d=3$. The formulas above then imply
\[
\begin{aligned}
\|w_j\|_{L^4(B)}
&\leq C_{d,R}\ell^{\alpha_d-1/4},\\
\|\nabla w_j\|_{L^4(B)}
&\leq C_{d,R} \ell^{\alpha_d+3/4}.
\end{aligned}
\]
Substituting $\ell\sim j^{1/(d-1)}$, we obtain that
\begin{align}
&\|w_j\|_{L^4(B)}\leq C_{d,R}\,j^{(\alpha_d-1/4)/(d-1)}\label{estw1},\\
&\|\nabla w_j\|_{L^4(B)}\le C_{d,R}\,j^{(\alpha_d+3/4)/(d-1)}\label{estw2},
\end{align}
in particular, we have
\[
\begin{aligned}
d=2:\qquad
\|w_j\|_{L^4(B_R)}
&\leq Cj^{-1/4},\|\nabla w_j\|_{L^4(B_R)}\leq Cj^{3/4},\\[2mm]
d=3:\qquad
\|w_j\|_{L^4(B_R)}
&\leq Cj^{-1/16},\|\nabla w_j\|_{L^4(B_R)}\leq Cj^{7/16}.
\end{aligned}
\]

To construct a real orthonormal basis of $L^2(B)$, we can use the set of normalized Dirichlet eigenfunctions of $-\Delta$ on $B$. We focus on the $\R^2$ case, i.e., consider
\begin{equation}\label{2deig}
-\Delta u=\lambda u\;\;\text{in}\;\; B\;\;\text{and}\;\; u=0\;\;\text{on}\;\;\partial B.
\end{equation}
Putting the eigenvalue problem in polar coordinates, it is a well-known result that the eigenvalue problem can be solved using Bessel's functions. Precisely, for $k\ge1$ and $l\ge1$, normalized eigenfunctions are
\[
\psi_{kl}^{(1)}(r,\theta)
=
\frac{\sqrt{2}}
{\sqrt{\pi}\,R\,|J_{k+1}(j_{kl})|}
J_k\left(\frac{j_{kl}}{R}r\right)\cos(k\theta),
\]
and
\[
\psi_{kl}^{(2)}(r,\theta)
=
\frac{\sqrt{2}}
{\sqrt{\pi}\,R\,|J_{k+1}(j_{kl})|}
J_k\left(\frac{j_{kl}}{R}r\right)\sin(k\theta).
\]
For \( k = 0,\, l \geq 1 \),
\[
\psi_{0l}(r,\theta)
=
\frac{1}
{\sqrt{\pi}\,R\,|J_1(j_{0l})|}
J_0\left(\frac{j_{0l}}{R}r\right).
\]
Furthermore, each has associated eigenvalue
\[
\lambda_{kl} = \left( \frac{j_{kl}}{R} \right)^2
\]
for $k\ge 0$, $l\ge 1$. Here, \( J_k \) is Bessel function of the first kind of order \( k \) and \( j_{kl} \) is the \( l \)-th positive zero of \( J_k \), i.e., \( J_k(j_{kl}) = 0 \). For the computational purpose, we will use the following formulas of Bessel's functions 
\[
J_k(x) = \sum_{m=0}^{\infty} \frac{(-1)^m}{m! \, (k+m)!} \left( \frac{x}{2} \right)^{k+2m}
\]
for $k\ge 0$. To calculate the zeros of $J_k$, one can use \(\texttt{scipy.special.jn\_zeros}\) in Python.

\section{Discretization-invariance property}\label{A2}

Here we will establish that the architecture of neural operators in Definition~\ref{def:neural-operator} is discretization-invariant. Let ${\mathcal K}\subset{\mathcal S}^{M}_{m,B'}$ be compact (with respect to $W^{2,\infty}(B)$-norm). For example, we may take $\mathcal K$ to be a subset of ${\mathcal S}_{m,B'}^M$ which is uniformly bounded in
$W^{3,\infty}(B)$. For $\gamma\in{\mathcal K}$, let $\gamma^{(L)}$ be any discretized version of $\gamma$, for example, by projecting $\gamma$ to finite-dimensional spaces, e.g., using FEM or spectral method, and
\begin{equation}\label{discrete}
\sup_{\gamma\in{\mathcal K}}\|\gamma^{(L)}-\gamma\|_{W^{2,\infty}(B)}\to 0\quad\mbox{as}\quad L\to\infty.
\end{equation}
Now let ${\mathcal G}\in \mathcal{NO}^{{\mathbf L}, \bfd, \sigma}_{N, J, \phi, \Phi}$, then we need to show that
\begin{equation}\label{DI-1}
\sup_{\gamma\in{\mathcal K}}\|{\mathcal G}(\gamma^{(L)})-{\mathcal G}(\gamma)\|_{\H_r}\to 0\quad\mbox{as}\quad L\to\infty.
\end{equation}

To prove \eqref{DI-1}, first of all, it is straightforward to observe that for all $1\le j,k\le J$, $\Gamma_{jk}^{(1)}(\gamma^{(L)})$ and $\Gamma_{jk}^{(2)}(\gamma^{(L)})$ converge to $\Gamma_{jk}^{(1)}(\gamma)$ and $\Gamma_{jk}^{(2)}(\gamma)$ uniformly for $\gamma\in{\mathcal K}$, respectively. Since the activation function $\sigma$ is continuous and the multiplications by the weight matrices $W^{(\ell)}$ are local operations, it is enough to consider the integral operator $K_{N_1}^{(\ell)}$.  To this end, let us denote $q^{(L)}:=(\gamma^{(L)})^{-1/2}\Delta(\gamma^{(L)})^{1/2}$, then we need to show that for any fixed $N\in\N$
\[
\|\Psi_{q^{(L)},N}[u]-\Psi_{q,N}[u]\|_{L^2(B)}\to 0
\]
as $L\to\infty$, which is obvious by combining \eqref{qN} and the fact $\|q^{(L)}-q\|_{L^\infty(B)}\to 0$.

\end{sloppypar}

\bibliographystyle{custom} 
\bibliography{ref}

@article{yan2017convergence,
	author = {Yan, Liang and Zhang, Yuan-Xiang},
	journal = {Inverse Problems},
	number = {12},
	pages = {125001},
	publisher = {IOP Publishing},
	title = {Convergence analysis of surrogate-based methods for Bayesian inverse problems},
	volume = {33},
	year = {2017}}

@article{majee2025mcmc,
	author = {Majee, Sudeb and Abhishek, Anuj and Strauss, Thilo and Khan, Taufiquar},
	journal = {Inverse Problems},
	number = {9},
	pages = {095013},
	publisher = {IOP Publishing},
	title = {MCMC-Net: accelerating Markov Chain Monte Carlo with neural networks for inverse problems},
	volume = {41},
	year = {2025}}

@article{sogge1988concerning,
	author = {Sogge, Christopher D.},
	doi = {10.1016/0022-1236(88)90081-X},
	journal = {Journal of Functional Analysis},
	number = {1},
	pages = {123--138},
	publisher = {Elsevier},
	title = {Concerning the {$L^p$} norm of spectral clusters for second-order elliptic operators on compact manifolds},
	volume = {77},
	year = {1988}}

@article{abhishek2024deeponet,
	author = {Abhishek, Anuj and Strauss, Thilo},
	journal = {arXiv preprint arXiv:2407.17182},
	title = {A DeepONet for inverting the Neumann-to-Dirichlet Operator in Electrical Impedance Tomography: An approximation theoretic perspective and numerical results},
	year = {2024}}

@article{AlbSant1,
	author = {Alberti, Giovanni S. and Santacesaria, Matteo},
	doi = {10.1017/fms.2019.31},
	fjournal = {Forum of Mathematics. Sigma},
	issn = {2050-5094},
	journal = {Forum Math. Sigma},
	mrclass = {35R30 (35J10 42C40 94A20)},
	mrnumber = {4016502},
	mrreviewer = {Dinh Nho H\`ao},
	pages = {Paper No. e35, 20},
	title = {Calder\'on's inverse problem with a finite number of measurements},
	url = {https://doi.org/10.1017/fms.2019.31},
	volume = {7},
	year = {2019}}

@article{AlbSant2,
	author = {Alberti, Giovanni S. and Santacesaria, Matteo},
	doi = {10.1007/s00205-021-01718-4},
	fjournal = {Archive for Rational Mechanics and Analysis},
	issn = {0003-9527,1432-0673},
	journal = {Arch. Ration. Mech. Anal.},
	mrclass = {35R30 (47J06 49N45 78A46 94A20)},
	mrnumber = {4359448},
	number = {1},
	pages = {1--31},
	title = {Infinite-dimensional inverse problems with finite measurements},
	url = {https://doi.org/10.1007/s00205-021-01718-4},
	volume = {243},
	year = {2022}}

@article{AlbSant3,
	author = {Alberti, Giovanni S. and Santacesaria, Matteo},
	doi = {10.1080/00036811.2020.1745192},
	fjournal = {Applicable Analysis. An International Journal},
	issn = {0003-6811,1563-504X},
	journal = {Appl. Anal.},
	mrclass = {35R30 (35J10 42C40 94A20)},
	mrnumber = {4446595},
	number = {10},
	pages = {3636--3654},
	title = {Calder\'on's inverse problem with a finite number of measurements {II}: independent data},
	url = {https://doi.org/10.1080/00036811.2020.1745192},
	volume = {101},
	year = {2022}}

@article{AlbSant4,
	author = {Alberti, Giovanni S. and Arroyo, \'Angel and Santacesaria, Matteo},
	doi = {10.1088/1361-6544/aca73d},
	fjournal = {Nonlinearity},
	issn = {0951-7715,1361-6544},
	journal = {Nonlinearity},
	mrclass = {35R30 (58C25)},
	mrnumber = {4521957},
	number = {1},
	pages = {734--808},
	title = {Inverse problems on low-dimensional manifolds},
	url = {https://doi.org/10.1088/1361-6544/aca73d},
	volume = {36},
	year = {2023}}

@article{feldman2019calderon,
	author = {Feldman, Joel and Salo, Mikko and Uhlmann, Gunther},
	journal = {Preliminary notes on the book in preparation},
	title = {The Calder{\'o}n Problem--An Introduction to Inverse Problems},
	volume = {30},
	year = {2019}}

@article{schillings2018convergence,
	author = {Schillings, Claudia and Stuart, Andrew M},
	journal = {Applicable Analysis},
	number = {1},
	pages = {107--123},
	publisher = {Taylor \& Francis},
	title = {Convergence analysis of ensemble Kalman inversion: the linear, noisy case},
	volume = {97},
	year = {2018}}

@article{schillings2017analysis,
	author = {Schillings, Claudia and Stuart, Andrew M},
	journal = {SIAM Journal on Numerical Analysis},
	number = {3},
	pages = {1264--1290},
	publisher = {SIAM},
	title = {Analysis of the ensemble Kalman filter for inverse problems},
	volume = {55},
	year = {2017}}

@article{iglesias2013ensemble,
	author = {Iglesias, Marco A and Law, Kody JH and Stuart, Andrew M},
	journal = {Inverse Problems},
	number = {4},
	pages = {045001},
	publisher = {IOP Publishing},
	title = {Ensemble Kalman methods for inverse problems},
	volume = {29},
	year = {2013}}

@article{scaop2025,
	author = {Furuya, Takashi and Jenn-Nan Wang},
	journal = {preprint},
	title = {Quantitative approximation theorems for scattering neural operators},
	year = {2025}}

@article{lu2021learning,
	author = {Lu, Lu and Jin, Pengzhan and Pang, Guofei and Zhang, Zhongqiang and Karniadakis, George Em},
	journal = {Nature machine intelligence},
	number = {3},
	pages = {218--229},
	publisher = {Nature Publishing Group UK London},
	title = {Learning nonlinear operators via DeepONet based on the universal approximation theorem of operators},
	volume = {3},
	year = {2021}}

@inproceedings{anandkumar2020neural,
	author = {Anandkumar, Anima and Azizzadenesheli, Kamyar and Bhattacharya, Kaushik and Kovachki, Nikola and Li, Zongyi and Liu, Burigede and Stuart, Andrew},
	booktitle = {ICLR 2020 Workshop on Integration of Deep Neural Models and Differential Equations},
	title = {Neural operator: Graph kernel network for partial differential equations},
	year = {2020}}

@article{lanthaler2023operator,
	author = {Lanthaler, Samuel},
	journal = {Journal of Machine Learning Research},
	number = {318},
	pages = {1--67},
	title = {Operator learning with PCA-Net: upper and lower complexity bounds},
	volume = {24},
	year = {2023}}

@article{abraham2020statistical,
	author = {Abraham, Kweku and Nickl, Richard},
	journal = {Mathematical Statistics and Learning},
	number = {2},
	pages = {165--216},
	title = {On statistical {C}alder{\'o}n problems},
	volume = {2},
	year = {2020}}

@article{de2022cost,
	author = {de Hoop, Maarten V and Huang, Daniel Zhengyu and Qian, Elizabeth and Stuart, Andrew M},
	journal = {arXiv preprint arXiv:2203.13181},
	title = {The cost-accuracy trade-off in operator learning with neural networks},
	year = {2022}}

@article{chen1995approximation,
	author = {Chen, Tianping and Chen, Hong},
	journal = {IEEE Transactions on Neural Networks},
	number = {4},
	pages = {904--910},
	publisher = {IEEE},
	title = {Approximation capability to functions of several variables, nonlinear functionals, and operators by radial basis function neural networks},
	volume = {6},
	year = {1995}}

@article{cybenko1989approximation,
	author = {Cybenko, George},
	journal = {Mathematics of control, signals and systems},
	number = {4},
	pages = {303--314},
	publisher = {Springer},
	title = {Approximation by superpositions of a sigmoidal function},
	volume = {2},
	year = {1989}}

@article{castro2024calderon,
	author = {Castro, Javier and Mu{\~n}oz, Claudio and Valenzuela, Nicol{\'a}s},
	journal = {Vietnam Journal of Mathematics},
	pages = {1--32},
	publisher = {Springer},
	title = {The Calder{\'o}n's problem via DeepONets},
	year = {2024}}

@article{marcati2023exponential,
	author = {Marcati, Carlo and Schwab, Christoph},
	journal = {SIAM Journal on Numerical Analysis},
	number = {3},
	pages = {1513--1545},
	publisher = {SIAM},
	title = {Exponential convergence of deep operator networks for elliptic partial differential equations},
	volume = {61},
	year = {2023}}

@article{kovachki2023neural,
	author = {Kovachki, Nikola B and Li, Zongyi and Liu, Burigede and Azizzadenesheli, Kamyar and Bhattacharya, Kaushik and Stuart, Andrew M and Anandkumar, Anima},
	journal = {J. Mach. Learn. Res.},
	number = {89},
	pages = {1--97},
	title = {Neural Operator: Learning Maps Between Function Spaces With Applications to {PDEs}.},
	volume = {24},
	year = {2023}}

@article{lanthaler2023nonlocal,
	author = {Lanthaler, Samuel and Li, Zongyi and Stuart, Andrew M},
	journal = {arXiv preprint arXiv:2304.13221},
	title = {The Nonlocal Neural Operator: Universal Approximation},
	year = {2023}}

@article{lu2019deeponet,
	author = {Lu, Lu and Jin, Pengzhan and Karniadakis, George Em},
	journal = {arXiv preprint arXiv:1910.03193},
	title = {Deeponet: Learning nonlinear operators for identifying differential equations based on the universal approximation theorem of operators},
	year = {2019}}

@article{lanthaler2022error,
	author = {Lanthaler, Samuel and Mishra, Siddhartha and Karniadakis, George E},
	journal = {Transactions of Mathematics and Its Applications},
	number = {1},
	pages = {tnac001},
	publisher = {Oxford University Press},
	title = {Error estimates for deeponets: A deep learning framework in infinite dimensions},
	volume = {6},
	year = {2022}}

@article{kratsios2024mixture,
	author = {Kratsios, Anastasis and Furuya, Takashi and Lassas, Matti and de Hoop, Maarten and others},
	journal = {arXiv preprint arXiv:2404.09101},
	title = {Mixture of Experts Soften the Curse of Dimensionality in Operator Learning},
	year = {2024}}

@article{lanthaler2023curse,
	author = {Lanthaler, Samuel and Stuart, Andrew M},
	journal = {arXiv preprint arXiv:2306.15924},
	title = {The curse of dimensionality in operator learning},
	year = {2023}}

@article{li2020fourier,
	author = {Li, Zongyi and Kovachki, Nikola and Azizzadenesheli, Kamyar and Liu, Burigede and Bhattacharya, Kaushik and Stuart, Andrew and Anandkumar, Anima},
	journal = {arXiv preprint arXiv:2010.08895},
	title = {Fourier neural operator for parametric partial differential equations},
	year = {2020}}

@article{tripura2022wavelet,
	author = {Tapas Tripura and Souvik Chakraborty},
	issn = {0045-7825},
	journal = {Computer Methods in Applied Mechanics and Engineering},
	pages = {115783},
	title = {Wavelet Neural Operator for solving parametric partial differential equations in computational mechanics problems},
	volume = {404},
	year = {2023}}

@article{chen2023learning,
	author = {Chen, Gengxiang and Liu, Xu and Meng, Qinglu and Chen, Lu and Liu, Changqing and Li, Yingguang},
	journal = {arXiv preprint arXiv:2302.08166},
	title = {Learning neural operators on {R}iemannian manifolds},
	year = {2023}}

@article{kovachki2021universal,
	author = {Kovachki, Nikola and Lanthaler, Samuel and Mishra, Siddhartha},
	journal = {Journal of Machine Learning Research},
	number = {290},
	pages = {1--76},
	title = {On universal approximation and error bounds for {F}ourier neural operators},
	volume = {22},
	year = {2021}}

@article{takamoto2022pdebench,
	author = {Takamoto, Makoto and Praditia, Timothy and Leiteritz, Raphael and MacKinlay, Daniel and Alesiani, Francesco and Pfl{\"u}ger, Dirk and Niepert, Mathias},
	journal = {Advances in Neural Information Processing Systems},
	pages = {1596--1611},
	title = {Pdebench: {A}n extensive benchmark for scientific machine learning},
	volume = {35},
	year = {2022}}

@article{kovachki2024operator,
	author = {Kovachki, Nikola B and Lanthaler, Samuel and Stuart, Andrew M},
	journal = {arXiv preprint arXiv:2402.15715},
	title = {Operator learning: {A}lgorithms and analysis},
	year = {2024}}

@article{deng2021convergence,
	author = {Beichuan Deng and Yeonjong Shin and Lu Lu and Zhongqiang Zhang and George Em Karniadakis},
	issn = {0893-6080},
	journal = {Neural Networks},
	pages = {411-426},
	title = {Approximation rates of DeepONets for learning operators arising from advection--diffusion equations},
	volume = {153},
	year = {2022}}

@article{li2019better,
	author = {Li , BoTang , Shanshan and Yu , Haijun},
	issn = {1991-7120},
	journal = {Communications in Computational Physics},
	number = {2},
	pages = {379--411},
	title = {Better Approximations of High Dimensional Smooth Functions by Deep Neural Networks with Rectified Power Units},
	volume = {27},
	year = {2019}}

@article{chen1995universal,
	author = {Chen, Tianping and Chen, Hong},
	journal = {IEEE transactions on neural networks},
	number = {4},
	pages = {911--917},
	publisher = {IEEE},
	title = {Universal approximation to nonlinear operators by neural networks with arbitrary activation functions and its application to dynamical systems},
	volume = {6},
	year = {1995}}

@article{bhattacharya2021model,
	author = {Bhattacharya, Kaushik and Hosseini, Bamdad and Kovachki, Nikola B and Stuart, Andrew M},
	journal = {The SMAI journal of computational mathematics},
	pages = {121--157},
	title = {Model reduction and neural networks for parametric PDEs},
	volume = {7},
	year = {2021}}

@article{li2020neural,
	author = {Li, Zongyi and Kovachki, Nikola and Azizzadenesheli, Kamyar and Liu, Burigede and Bhattacharya, Kaushik and Stuart, Andrew and Anandkumar, Anima},
	journal = {arXiv preprint arXiv:2003.03485},
	title = {Neural operator: Graph kernel network for partial differential equations},
	year = {2020}}

@inproceedings{bonev2023spherical,
	author = {Bonev, Boris and Kurth, Thorsten and Hundt, Christian and Pathak, Jaideep and Baust, Maximilian and Kashinath, Karthik and Anandkumar, Anima},
	booktitle = {International conference on machine learning},
	organization = {PMLR},
	pages = {2806--2823},
	title = {Spherical fourier neural operators: Learning stable dynamics on the sphere},
	year = {2023}}

@article{gupta2021multiwavelet,
	author = {Gupta, Gaurav and Xiao, Xiongye and Bogdan, Paul},
	journal = {Advances in neural information processing systems},
	pages = {24048--24062},
	title = {Multiwavelet-based operator learning for differential equations},
	volume = {34},
	year = {2021}}

@article{chen2023deep,
	author = {Ke Chen and Chunmei Wang and Haizhao Yang},
	issn = {2835-8856},
	journal = {Transactions on Machine Learning Research},
	title = {Deep Operator Learning Lessens the Curse of Dimensionality for {PDE}s},
	year = {2023}}

@article{furuya2024quantitative,
	author = {Furuya, Takashi and Taniguchi, Koichi and Okuda, Satoshi},
	journal = {arXiv preprint arXiv:2410.02151},
	title = {Quantitative Approximation for Neural Operators in Nonlinear Parabolic Equations},
	year = {2024}}

@article{dHKLN25,
	author = {de Hoop, Maarten V. and Kovachki, Nikola B. and Lassas, Matti and Nelsen, Nicholas H.},
	title = {Extension and neural operator approximation of the electrical impedance tomography inverse map},
	journal = {arXiv preprint arXiv:2511.20361},
	year = {2025},
	doi = {10.48550/arXiv.2511.20361},
	url = {https://arxiv.org/abs/2511.20361}}

\end{document}